\documentclass[reqno,11pt]{amsart}

\usepackage{amsthm, amsmath, amssymb, stmaryrd}
\usepackage[utf8]{inputenc}
\usepackage[T1]{fontenc}

\usepackage[hyphenbreaks]{breakurl}
\usepackage[hyphens]{url}
\usepackage{systeme}
\usepackage[shortlabels]{enumitem}
\usepackage[hidelinks]{hyperref}
\usepackage{microtype}
\usepackage{subcaption}
\usepackage{pgfplots}
\usepackage{cite}
\pgfplotsset{compat=1.18}

\usepackage{bm}
\usepackage[margin=1in]{geometry}

\usepackage[textsize=scriptsize,backgroundcolor=orange!5]{todonotes}

\usepackage[noabbrev,capitalize,sort,nameinlink]{cleveref}
\crefname{equation}{}{}
\numberwithin{equation}{section}

\usepackage{mathtools}
\newtheorem{theorem}{Theorem}[section]
\newtheorem{proposition}[theorem]{Proposition}
\newtheorem{lemma}[theorem]{Lemma}

	\crefname{claim}{Claim}{Claims}

\theoremstyle{definition}
\newtheorem{definition}[theorem]{Definition}

\newtheorem{example}[theorem]{Example}

\theoremstyle{remark}
\newtheorem*{remark}{Remark}

\newcommand{\abs}[1]{\left\lvert#1\right\rvert}

\makeatletter
\renewcommand{\paragraph}{%
  \@startsection{paragraph}{4}{0pt}%
    {1.5ex}
    {-1em}
    {\normalfont\normalsize\bfseries}%
}
\makeatother

\DeclareMathOperator{\ch}{char}

\newcommand{\FF}{\mathbb{F}}
\newcommand{\RR}{\mathbb{R}}
\newcommand{\NN}{\mathbb{N}}
\newcommand{\ZZ}{\mathbb{Z}}
\newcommand*{\PP}{\mathbb{P}} 

\newcommand{\cL}{\mathcal L}

\newcommand{\cP}{\mathcal P}

\newcommand{\fm}{\mathfrak{m}}
\newcommand{\fp}{\mathfrak{p}}

\newcommand{\Hasse}{\mathsf H} 

\newcommand{\cO}{\mathcal O}

\newlength{\hght}

\makeatletter
\newcommand\thankssymb[1]{\textsuperscript{\@fnsymbol{#1}}}
\makeatother

\author[Hung-Hsun Hans Yu]{Hung-Hsun Hans Yu\thankssymb{1}}

\thanks{\thankssymb{1}Department of Mathematics, Princeton University, Princeton, NJ 08544\@. Email: {\tt hansonyu@princeton.edu}}

\title{Two proofs and a tight Szemer\'edi--Trotter theorem}
\begin{document}
\begin{abstract}
The Szemer\'edi--Trotter theorem, a fundamental theorem in incidence geometry, states that $n$ points and $m$ lines in $\mathbb{R}^2$ form at most $O(m+n+m^{2/3}n^{2/3})$ incidences.
Recently, Lewko and independently Miao--Xie extended the Szemer\'edi--Trotter theorem to arbitrary fields with a correction term $O(\frac{mn}{\mathfrak{p}})$ if the field has positive characteristic $\fp$.

In this paper, we provide some geometric intuition for the two proofs and compare them.
Moreover, we show how both proofs can be improved using Hasse derivatives.
As a consequence, we obtain a tight upper bound over any algebraically closed field.
To be specific, when the field has positive characteristic $\mathfrak{p}$, $n\leq m$ and $\mathfrak{p}^{2(k-1)}\leq m<\mathfrak{p}^{2k}$ for some positive integer $k$, the number of incidences can be bounded by
\[O\left(\mathfrak{p}^{k-1}n+m+m^{2/3}n^{2/3}+\frac{mn}{\mathfrak{p}^{k}}\right).\]
\end{abstract}
\maketitle
\vspace{-1cm}
\section{Introduction}
Let $\FF$ be any field, $\cP$ be a set of points and $\cL$ be a set of lines in $\FF^2$.
A fundamental problem in incidence geometry is to understand the number of incidences $I(\cP,\cL)$ between $\cP$ and $\cL$, which is the number of pairs $(p,\ell)\in \cP\times \cL$ with $p\in \ell$.

For $\FF = \RR$, the Szemer\'edi--Trotter theorem states that the number of incidences can be upper bounded, up to a constant, by $\abs{\cP}+\abs{\cL}+\abs{\cP}^{2/3}\abs{\cL}^{2/3}$.
Previous proofs heavily depend on the topology of the real numbers and do not generalize to other fields easily.
The methods do not generalize for a good reason, as the bound does not hold verbatim over any field of positive characteristic: indeed, taking the $q^2$ points and $q(q+1)$ lines in $\FF_q^2$ gives $q^2(q+1)$ incidences, whereas the Szemer\'edi--Trotter bound gives only $O(q^{8/3})$.
However, any such constructions contain many points and lines compared to the characteristic of $\FF$, and it has been a question whether the Szemer\'edi--Trotter bound holds for much smaller sets of points and lines (for example, see the remark after Theorem 6.2 in \cite{BKT04}).

Recently, Lewko \cite{Lew26} and independently Miao and Xie \cite{MX26} proved that a correction term of $\abs{\cP}\abs{\cL}/\ch\FF$ is sufficient to make the Szemer\'edi--Trotter theorem hold for fields of positive characteristic, answering the previous question affirmatively.
Crucially, neither proof relies on the topology of the underlying field, which distinguishes it from all previous proofs.
For more background and consequences of the theorem, we refer the readers to their papers \cite{Lew26,MX26}.

For our convenience, we set $\fp$ to be positive infinity if the characteristic of $\FF$ is zero, and set $\fp=\ch\FF$ otherwise.
Their result can be stated as follows.
\begin{theorem}[Szemer\'edi--Trotter over arbitrary fields \cite{Lew26,MX26}]\label{thm:st}
    Let $\cP$ be a set of points and $\cL$ be a set of lines in $\FF^2$.
    Then
    \[I(\cP,\cL)\lesssim \abs{\cP}+\abs{\cL}+\abs{\cP}^{2/3}\abs{\cL}^{2/3}+\frac{\abs{\cP}\abs{\cL}}{\fp}.\]
\end{theorem}

Lewko \cite{Lew26} showed that all four terms are necessary in the regime $\abs{\cP},\abs{\cL}\leq \fp^2$, leaving the regime $\max(\abs{\cP},\abs{\cL})>\fp^2$ open.

The purpose of this paper is twofold: first, to explain and compare the two proofs of \cref{thm:st}; second, to show how Lewko's proof can be modified to obtain the following improvement that is tight in all regimes when $\FF$ is algebraically closed.
We set $\fp^0=1$ if $\fp=\infty$ throughout the paper.

\begin{theorem}[Improved Szemer\'edi--Trotter over arbitrary fields]\label{thm:main} Let $\cP$ be a set of points and $\cL$ be a set of lines, and suppose that $\abs{\cP}\leq \abs{\cL}$.
Suppose that $k$ is a positive integer such that $\fp^{2(k-1)}\leq \abs{\cL}<\fp^{2k}$.
Then
\[I(\cP,\cL)\lesssim \fp^{k-1}\abs{\cP}+\abs{\cL}+\abs{\cP}^{2/3}\abs{\cL}^{2/3}+\frac{\abs{\cP}\abs{\cL}}{\fp^{k}}.\]
    
\end{theorem}

\subsection{Comparison of the two proofs}
The two proofs by Lewko \cite{Lew26} and Miao--Xie \cite{MX26} share a lot in common at the beginning.
They both proceed by lifting the lines from $\FF^2$ to $\FF^3$ and constructing polynomials of small degrees vanishing on all the lines.
However, the two proofs take very different routes after this.

In fact, instead of directly proving \cref{thm:st}, Miao and Xie bound the number of $r$-rich lines for a certain range of $r$.
Here, a line is \emph{$r$-rich} for a set of points $\cP$ if it contains at least $r$ of the points in the set.
\cref{thm:st} follows from their bound on the number of $r$-rich lines via a standard computation.
The bound on the number of $r$-rich lines that Miao and Xie prove is the following.
\begin{theorem}\label{thm:mx-bound-rich-line}
    Let $\cP$ be a set of points in $\FF^2$.
    Assume that $r$ is a positive integer satisfying $r\geq \max(2,\frac{12\abs{\cP}}{\fp})$.
    Let $\cL_r$ be the set of $r$-rich lines for $\cP$.
    Then
    \[\abs{\cL_r}\leq \frac{16\abs{\cP}^2}{r^3}+\frac{44\abs{\cP}}{r}.\]
\end{theorem}
We remark that Miao and Xie did not make the constants explicit in their argument.
We choose to optimize the constant for $\abs{\cP}^2/r^3$ for readers to compare with Lewko's bound.

On the other hand, Lewko directly deals with the number of incidences and proves the following parametrized inequality.
\begin{theorem}\label{thm:lewko-st-with-param}
    Let $\cP$ be a set of points and $\cL$ be a set of lines in $\FF^2$.
    For any real number $ 1\leq d<\fp$, we have
    \[I(\cP,\cL)\leq \abs{\cP}+\abs{\cL}+d(d-1)+\frac{2\abs{\cP}\abs{\cL}}{d}.\]
\end{theorem}
We will show that by choosing $d$ appropriately, the right hand side can be upper bounded by $\abs{\cP}+\abs{\cL}+3\abs{\cP}^{2/3}\abs{\cL}^{2/3}+\frac{2\abs{\cP}\abs{\cL}}{\fp}$.
Again, Lewko did not make the constants explicit in the proof, and we make the constants explicit for comparison.

As can be seen from the constants in the theorem statements above, the computation in Lewko's proof is simpler than the computation in Miao and Xie's proof.
In addition, the version of B\'ezout's theorem Lewko uses is more common than the version (\cref{lemma:bihom-bezout}) of B\'ezout's theorem Miao and Xie use.
However, Lewko's proof contains a step (\cref{lemma:incidence-to-mult}) that is hard to motivate before seeing Miao and Xie's proof.
We will thus present Miao and Xie's proof first, although some readers might find Lewko's proof easier to understand.

\subsection{Tight point-line incidence}
Characteristics show up in both proofs for the same reason---when $\fp<\infty$, it is possible for a nonzero polynomial to have formal derivatives all equal to the zero polynomial.
For example, in $\FF[x]$, the formal derivative of the polynomial $x^{\fp}$ is $\fp x^{\fp-1}=0$.
This kind of obstruction is necessary because of the existence of the finite field construction mentioned earlier that breaks the usual Szemer\'edi--Trotter bound.

In applications of polynomial methods, this issue is usually dealt with using the Hasse derivative (see \cite{KLSS11,DKSS13,CI20,BC21,DDL21,TYZ22,CY26-2} for examples).
The same approach can be applied to both Lewko's proof and Miao and Xie's proof.
By carefully controlling the degrees of the polynomials and the orders of the Hasse derivatives used, the $\frac{\abs{\cP}\abs{\cL}}{\fp}$ term can be improved at the cost of a larger coefficient for $\abs{\cP}$.
Unfortunately, when going from the bound on rich lines to the bound on the number of incidences, Miao and Xie's proof produces an extra $\log \fp$ factor for the coefficient of $\abs{\cP}$.

We will also see that assuming $\FF$ is algebraically closed (and thus contains $\FF_{\fp^m}$ for all $m\in\NN$ if $\fp<\infty$), \cref{thm:main} is tight up to a constant for all regimes.
This resolves the point-line incidence problem over algebraically closed fields, leaving some gaps open for finite fields.
We will discuss this in the concluding remarks.

The bound in \cref{thm:main} can be compared with the K\H{o}v\'ari--S\'os--Tur\'an bound 
\[I(\cP,\cL)\lesssim \abs{\cL}+\abs{\cP}\abs{\cL}^{1/2}.\]
Suppose that $\fp<\infty$ and $\abs{\cL} = \fp^{2(k-1+\theta)}$, where $\theta\in [0,1)$.
Then the bound in \cref{thm:main} can be rephrased as
\[I(\cP,\cL)\lesssim \abs{\cL}+\fp^{-\min(\theta,1-\theta)}\abs{\cP}\abs{\cL}^{1/2}+\abs{\cP}^{2/3}\abs{\cL}^{2/3},\]
showing that this is an improvement by a factor of at most $\fp^{1/2}$.
In particular, the improvement is strongest when $\abs{\cL}$ is ``stuck in the middle'' between consecutive even powers of $\fp$.

It is likely that \cref{thm:main} may lead to improvements in many applications of Szemer\'edi--Trotter's theorem over finite fields, many of which are recorded in Lewko's and Miao and Xie's work \cite{Lew26,MX26}.
However, as the improvement is at most a factor of $\fp^{1/2}$ compared to the K\H{o}v\'ari--S\'os--Tur\'an bound, we do not pursue this direction and leave this for interested readers.
\subsection{Structure of the paper}
In \cref{sec:lift}, we first lay out the lifting framework that both proofs work in.
In \cref{sec:mx}, we present Miao and Xie's proof, which is followed by an independent presentation of Lewko's proof in \cref{sec:lewko}.
We then show how both proofs can be modified using Hasse derivatives and derive \cref{thm:main} using Lewko's proof in \cref{sec:main},  where we also prove the tightness of the bound and compare it with some other known bounds.
\cref{sec:conclude} ends the paper with several concluding remarks.

\paragraph{AI Statement} ChatGPT 6 Pro was used to summarize Lewko's paper \cite{Lew26} for the author and produce all the figures in this note. 
It was also used to help with literature search and check for grammar mistakes and typos.
During this process, ChatGPT 6 Pro correctly pointed out that \cref{thm:modified-mx-bound-rich-line} does not directly imply \cref{thm:main} because of a logarithmic loss the author previously overlooked.
ChatGPT-6 Astra was used to produce a construction in \cref{subsec:tightness}, although the construction was independently communicated by Dmitrii Zakharov to the author.
All text in this note was written by the author, and the author takes full responsibility for any mistakes.

\section{Lifting to three-dimensional space}\label{sec:lift}
In this section, we explain how the two-dimensional configuration is lifted to $\FF^3$ and outline the proof strategy.
This lifting is the heart of the argument.

In the remainder of the paper, unless specified otherwise, we will assume $\FF$ is algebraically closed by taking field extensions.
Then by an appropriate change of coordinates, we will assume that all lines in $\cL$ are not vertical.

Any line $\ell$ in $\cL$ can be parametrized as $\{(t,at+b):t\in \FF\}$ for some $(a,b)\in \FF^2$.
For any such line, we will lift it to $\Tilde{\ell}\subseteq \FF^3$ by $\Tilde{\ell}=\{(t,at+b,a):t\in \FF\}$.
Let $\Tilde{\cL}$ be the set of the lifted lines in $\FF^3$.
One should think of $\FF^3$ as $\FF^2\times \FF$, where the first component is the \emph{spatial component} and the second component is the \emph{slope component}.
A \emph{slice} is the plane $\FF^2\times \{s\}\subseteq \FF^3$ for some fixed slope $s$.
What we are doing is simply lifting the lines to the correct slice according to their slope, and so the lines within a particular slice are all parallel to one another.
In both proofs, we will treat the two components really differently, and this will be emphasized again when it is used.

By doing so, we have ``separated'' the incidences occurring at a particular point $p\in \cP$.
Indeed, if $p$ is incident to $r$ lines in $\cL$, then the line $\ell_p\coloneqq \{(p,t):t\in \FF\}$ intersects with $r$ lines from $\Tilde{\cL}$ at $r$ distinct points.
The separation of the incidences is going to be useful for Lewko's proof.

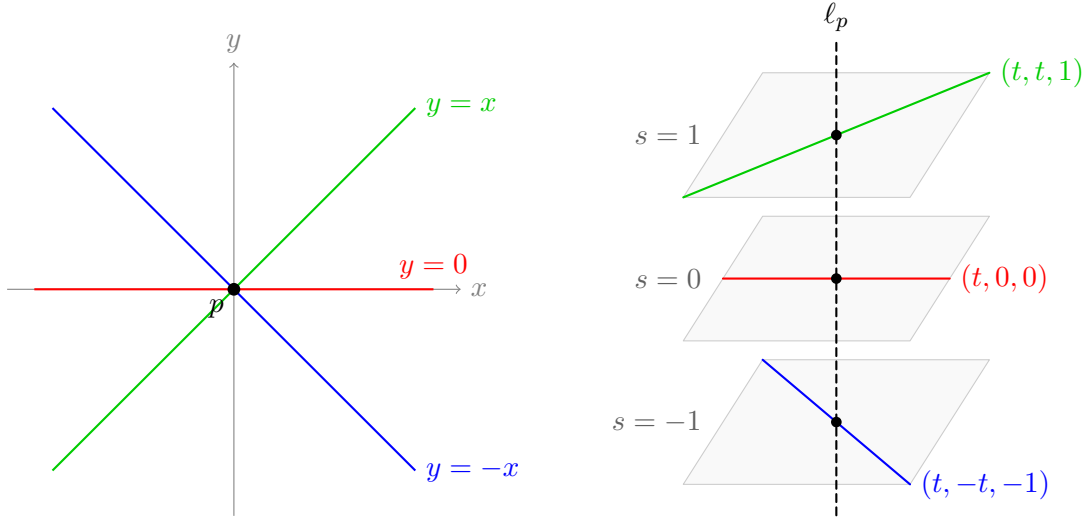
\begin{figure}[h]
    \centering
    \begin{subfigure}[b]{0.4\textwidth}
        \begin{tikzpicture}[scale=1.2]
    \draw[->, gray] (-2.5,0) -- (2.5,0) node[right] {$x$};
    \draw[->, gray] (0,-2.5) -- (0,2.5) node[above] {$y$};

    \draw[thick, blue] (-2,2) -- (2,-2) node[right] {$y=-x$};
    \draw[thick, red] (-2.2,0) -- (2.2,0) node[above] {$y=0$};
    \draw[thick, green!80!black] (-2,-2) -- (2,2) node[right] {$y=x$};

    \fill (0,0) circle (2pt) node[below left] {$p$};
\end{tikzpicture}
    \end{subfigure}
    \quad\quad\quad
    \begin{subfigure}[b]{0.4\textwidth}
    \begin{tikzpicture}[
    x={(1cm,0cm)},
    y={(0.35cm,0.55cm)},
    z={(0cm,1.9cm)},
    line cap=round,
    line join=round,
    lifted/.style={thick},
    level/.style={draw=gray!40, fill=gray!5, thin}
]
    \foreach \a in {-1,0,1} {
        \path[level]
            (-1.5,-1.5,\a) -- (1.5,-1.5,\a) --
            (1.5,1.5,\a) -- (-1.5,1.5,\a) -- cycle;

        \node[left, text=gray!75!black]
            at (-1.65,0,\a) {$s=\a$};
    }

    \draw[thick, densely dashed]
        (0,0,-1.65) -- (0,0,1.65)
        node[above] {$\ell_p$};

    \draw[lifted, blue]
        (-1.5,1.5,-1) -- (1.5,-1.5,-1)
        node[right] {$(t,-t,-1)$};

    \draw[lifted, red]
        (-1.5,0,0) -- (1.5,0,0)
        node[right] {$(t,0,0)$};

    \draw[lifted, green!80!black]
        (-1.5,-1.5,1) -- (1.5,1.5,1)
        node[right] {$(t,t,1)$};

    \foreach \a in {-1,0,1}
        \fill (0,0,\a) circle[radius=2pt];

\end{tikzpicture}

    \end{subfigure}
    \caption{Lifting three lines through the origin.}
\end{figure}

A more important advantage of this lifting is that we can measure degrees of a polynomial in two different ways.
For any $f\in \FF[x,y,s]$, we will measure its \emph{spatial degree} $\deg_{x,y}f$ separately from its \emph{slope degree} $\deg_s f$.
This extra degree of freedom will be crucial in both proofs.
We briefly remark that the idea of measuring degrees asymmetrically has already shown up in other arguments using polynomial methods such as \cite{BC21}.

For both proofs, it will be useful to think about the family of lines through a fixed point and what they lift up to in $\FF^3$.
Suppose that $p=(p_x,p_y)$, and think of $s$ as a moving element in $\FF$.
As $(p,s)\in \FF^3$ runs through the line $\ell_p$, think of the line $\ell(p,s)\coloneqq\{(p_x+t,p_y+st,s):t\in \FF\}$ as a moving line in the $s$-th slice through the point $(p,s)$.
The family of the lines $\{\ell(p,s):s\in \FF\}$ sweeps out a degree-$2$ surface that is usually called a \emph{regulus}, which we will denote as $R_p$.

\begin{figure}[h]
    \centering
    \begin{tikzpicture}[
    x={(1cm,0cm)},
    y={(0.35cm,0.55cm)},
    z={(0cm,1.9cm)},
    line cap=round,
    line join=round,
    mesh/.style={cyan!40!gray, thin},
    boundary/.style={cyan!45!gray, semithick},
    lifted/.style={thick}
]
    %
    \shade[top color=cyan!7, bottom color=cyan!20]
        (-1.5, 1.875,-1.25) -- ( 1.5,-1.875,-1.25) --
        ( 1.5, 1.875, 1.25) -- (-1.5,-1.875, 1.25) -- cycle;

    \foreach \s in {-1.25,-1,...,1.25} {
        \draw[mesh]
            (-1.5,{-1.5*\s},\s) -- (1.5,{1.5*\s},\s);
    }

    \foreach \t in {-1.5,-1.25,...,1.5} {
        \draw[mesh]
            (\t,{-1.25*\t},-1.25) -- (\t,{1.25*\t},1.25);
    }

    \draw[boundary]
        (-1.5, 1.875,-1.25) -- ( 1.5,-1.875,-1.25) --
        ( 1.5, 1.875, 1.25) -- (-1.5,-1.875, 1.25) -- cycle;

    \draw[thick, densely dashed]
        (0,0,-1.9) -- (0,0,1.9)
        node[above] {$\ell_p$};

    \draw[lifted, blue]
        (-1.7,1.7,-1) -- (1.7,-1.7,-1)
        node[right] {$(t,-t,-1)$};

    \draw[lifted, orange]
        (-1.7,-0.85,0.5) -- (1.7, 0.85, 0.5)
        node[right] {$\ell(p,s)$};

    \draw[lifted, red]
        (-1.7,0,0) -- (1.7,0,0)
        node[right] {$(t,0,0)$};

    \draw[lifted, green!80!black]
        (-1.7,-1.7,1) -- (1.7,1.7,1)
        node[right] {$(t,t,1)$};

    \foreach \s in {-1,0,1}
        \fill (0,0,\s) circle[radius=2pt];
    \fill (0,0,0.5) circle[radius=2pt] node[right]{$(p,s)$};

    \node[above left, text=cyan!35!black]
        at (-1.5,-1.875,1.25) {$R_p:y=sx$};

\end{tikzpicture}
    \caption{The moving line $\ell(p,s)$ and the surface $R_p$ it sweeps out for $p=(0,0)$. The regulus $R_p$ contains two one-dimensional families of lines as shown in the diagram, although only the family parametrized by $\ell(p,s)$ will be relevant.}
\end{figure}
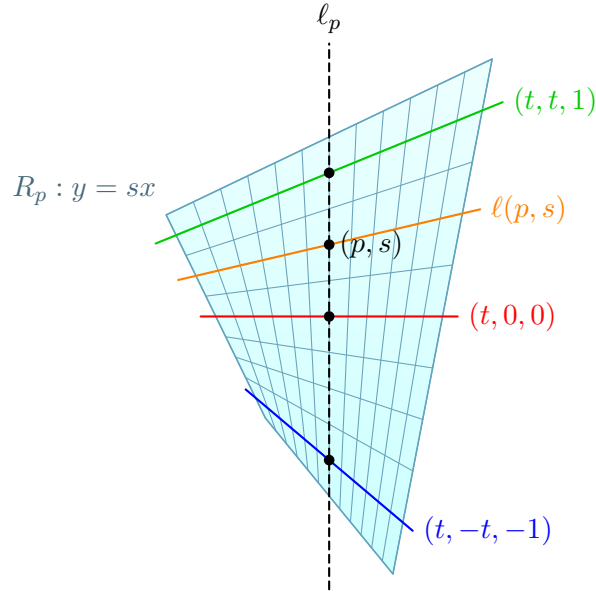

In both proofs, we will be measuring the multiplicity of some polynomial $f\in \FF[x,y,s]$ along the lines $\ell(p,s)$.
For our convenience, define $\mu_f(p,s)$ to be the multiplicity of $f$ at $(p,s)$ along the line $\ell(p,s)$.
Geometrically, this is the order that the line $\ell(p,s)$ is tangent to the surface defined by $f$ at $(p,s)$.
Algebraically, it is the largest non-negative integer $m$ so that $t^m\mid f(p_x+t,p_y+st,s)$ as polynomials in $\FF[t]$.
Note that if the line $\ell(p,s)$ is contained in the zero locus of $f$, then we set $\mu_f(p,s) = \infty$.

\section{Miao and Xie's proof}\label{sec:mx}
In this section, we present Miao and Xie's proof of \cref{thm:mx-bound-rich-line}.
We will also show that it implies \cref{thm:st}.
\subsection{Vanishing lemma}
Let $\cP\subseteq \FF^2$ be a set of points, and let $r$ be any positive integer.
Let $\cL_r$ be the set of $r$-rich lines for $\cP$, and let $\Tilde{\cL_r}$ be the set of lifted lines.
The proof begins by finding a polynomial $f$ of controlled degrees that vanishes identically on all lines in $\Tilde{\cL_r}$.
This will be done by requiring $f$ to vanish ``around $\ell_p$ to high order'' for each $p\in \cP$, which we make precise below.

Suppose that $\ell\in \cL_r$ passes through the points $p_1,\ldots,p_r\in \cP$.
Let $s_0$ be the slope of $\ell$.
Then $(p_1,s_0),\ldots, (p_r,s_0)$ are points on the lifted line $\Tilde{\ell}$.
If $r>\deg_{x,y}f$, it thus suffices to require $f$ to vanish at each $(p_i,s_0)$.
For larger spatial degrees, we can require $f$ to vanish to higher multiplicities at each $(p_i,s_0)$ along the line $\Tilde{\ell}$.
To be more specific, if $f|_{\Tilde{\ell}}$ vanishes at each $(p_i,s_0)$ to multiplicity $m$ and $rm>\deg_{x,y}f$, then $f$ has to be identically zero on $\Tilde{\ell}$.
Since $\Tilde{\ell} = \ell(p_i,s_0)$ for each $i\in[r]$, the requirement is the same as requiring $\mu_f(p_i,s_0)\geq m$.

Thus, to make $f$ vanish identically on all lines in $\Tilde{\cL_r}$, it suffices to ask $\mu_f(p,s)\geq m$ for every $p\in \cP$ when there is a line in $\cL_r$ through $p$ with slope $s$.
It turns out that we can simply ask $\mu_f(p,s)\geq m$ to hold for all $s\in \FF$ for every $p\in\cP$.
We record the observations as the lemma below.

\begin{lemma}[Vanishing lemma]\label{lemma:mx-vanishing-lemma}
    Let $\cP\subseteq \FF^2$ be a set of points, and let $r$ be any positive integer.
    Let $\cL_r$ be the set of $r$-rich lines for $\cP$, and let $\Tilde{\cL}_r$ be the set of lifted lines.
    
    Suppose $f\in\FF[x,y,s]$ and $m\in \NN$ are such that $rm>\deg_{x,y}f$ and $\mu_f(p,s)\geq m$ for every $s\in \FF$ and $p\in \cP$. Then $f$ is identically zero on each line in $\Tilde{\cL_r}$.
\end{lemma}
\begin{proof}
    Let $\ell\in \cL_r$ be a line passing through distinct points $p_1,\ldots, p_r\in \cP$.
    Suppose that the line $\ell$ can be parametrized as $\{(t,at+b):t\in \FF\}$ and $p_i = (t_i,at_i+b)$ for every $i\in[r]$.
    By definition, we have $\Tilde{\ell} = \{(t,at+b,a):t\in \FF\}$.

    By the assumption, we have $\mu_f(p_i,a)\geq m$ for each $i\in [r]$.
    By definition, this means \[t^m\mid f(t_i+t,a(t_i+t)+b,a),\]
    which is the same as 
    \[(t-t_i)^m\mid f(t,at+b,a).\]
    In other words, the univariate polynomial $f(t,at+b,a)$ in $t$ has a root at $t_i$ of multiplicity at least $m$ for each $i\in [r]$.
    Therefore it has at least $rm$ roots counted with multiplicities.
    On the other hand, $\deg_t f(t,at+b,a)\leq \deg_{x,y}f<rm$, which shows that $f(t,at+b,a)$ is the zero polynomial in $t$.
    Therefore $f$ is identically zero on $\Tilde{\ell}$.
\end{proof}

\subsection{Parameter counting}
Our goal now is to find a polynomial $f\in \FF[x,y,s]$ with small degrees that satisfies the condition in \cref{lemma:mx-vanishing-lemma}.
This is done by a usual parameter counting argument, although we have to count carefully.
For example, for a fixed $p\in \cP$, suppose we would like to have $\mu_f(p,s)\geq 1$ for all $s\in \FF$.
This is the same as requiring $f(p,s)=0$ for all $s\in \FF$, although this is too wasteful as $\FF$ is infinite.
Instead, we can expand $f(p,s)$ as a polynomial in $s$ of degree at most $\deg_sf$, and we just need to require that the $(\deg_s f+1)$ coefficients of the polynomial are zero.
This allows us to put finitely many linear conditions on $f$ to ensure that $f(p,s)=0$ for all $s\in \FF$.
The full proof uses the same idea to deal with higher multiplicities.

\begin{lemma}[Parameter counting for Miao--Xie]\label{lemma:mx-parameter-counting}
    Let $\cP$ be a set of points in $\FF^2$.
    Let $m,d,e$ be any non-negative integers such that 
    \[\binom{d+2}{2}(e+1)>\left(e+\frac{m+1}{2}\right)m\abs{\cP}.\]
    Then there exists a nonzero polynomial $f\in \FF[x,y,s]$ with $\deg_{x,y}f\leq d$, $\deg_s f\leq e$ and $\mu_f(p,s)\geq m$ for any $p\in \cP$ and $s\in \FF$.
\end{lemma}
\begin{proof}
    The space of polynomials with spatial degree at most $d$ and slope degree at most $e$ has dimension exactly $\binom{d+2}{2}(e+1)$.
    On the other hand, to ensure $\mu_f(p,s)\geq m$ for a fixed $p\in \cP$ and for all $s\in \FF$, we expand 
    \[f(p_x+t,p_y+st,s) = \sum_{i=0}^{d}c_i(s)t^i\]
    where each $c_i\in \FF[s]$.
    Then we just need to make sure that $c_i(s)\equiv 0$ for every $0\leq i\leq m-1$.
    To count the number of linear constraints needed to ensure $c_i(s)\equiv 0$, we need to understand the degree of $c_i$.

    Note that each term of $f(p_x+t,p_y+st,s)$ is of the form $(p_x+t)^{\alpha}(p_y+st)^{\beta}s^{\gamma}$ with $\alpha+\beta\leq d$ and $\gamma\leq e$.
    When we expand $(p_x+t)^{\alpha}(p_y+st)^{\beta}$, the degree of $s$ is always at most the degree of $t$.
    Multiplying by $s^{\gamma}$ raises the degree of $s$ by at most $e$, showing that $\deg c_i\leq e+i$.

    Therefore to require $\mu_f(p,s)\geq m$ for all $s\in \FF$ for a single $p$, we just need to impose
    \[\sum_{i=0}^{m-1}(e+i+1) = \frac{(2e+m+1)m}{2} = \left(e+\frac{m+1}{2}\right)m\]
    linear conditions on $f$.
    Multiplying by $\abs{\cP}$, we see that the number of linear conditions we put on $f$ is smaller than the dimension of the space of polynomials.
    Therefore there must be a nonzero polynomial $f$ in the space satisfying all the linear conditions, which in turn guarantees that $\mu_f(p,s)\geq m$ for any $p\in\cP$ and $s\in \FF$.
\end{proof}

\subsection{B\'ezout's theorem for bihomogeneous polynomials}
By combining \cref{lemma:mx-vanishing-lemma} with \cref{lemma:mx-parameter-counting}, we have obtained a polynomial $f$ that vanishes on the set of lifted rich lines $\Tilde{\cL_r}$.
We will use this to bound the size of $\Tilde{\cL_r}$ and thus the size of $\cL_r$.

The first observation is that if $f$ is identically zero on a lifted line $\Tilde{\ell}$, then the derivative of $f$ along $\Tilde{\ell}$ must also be zero.
The directional derivatives for all lifted lines can be written as $(\partial_x+s\partial_y)f$, which can be seen by a direct computation.
The geometric intuition is that if we work over $\RR$, the vector field $(1,s,0)$ in $\RR^3$ integrates to exactly the family of lifted lines from $\RR^2$, and so differentiating along the lifted lines is the same as differentiating along the vector field.
With this observation, we denote by $D$ the derivative operator $\partial_x+s\partial_y$.

Since both $f$ and $Df$ vanish on $\Tilde{\cL_r}$, a natural attempt to bound $\abs{\Tilde{\cL_r}}$ is to apply B\'ezout's theorem to the polynomials $f$ and $Df$.
However, the common factors of $f$ and $Df$ need to be dealt with separately.
Those will be dealt with in the next subsection.
For now, assume that $f$ and $Df$ are coprime, which shows that their common zero locus is at most $1$-dimensional.

Now a na\"ive application of B\'ezout's theorem shows that the intersection of the zero loci of $f$ and $Df$ has degree at most $\deg_{x,y,s}f\deg_{x,y,s}Df\leq (\deg_{x,y,s}f)^2$.
Therefore $\abs{\Tilde{\cL_r}}$ is bounded by the same number as well.
However, this bound turns out to be lossy when the spatial degree of $f$ is much larger or much smaller than the slope degree of $f$.
We will instead prove the following bound.

\begin{lemma}\label{lemma:mx-improved-bezout}
    Let $f,g$ be coprime polynomials in $\FF[x,y,s]$.
    Then there are at most $\deg_sf\deg_{x,y}g+\deg_{x,y}f\deg_sg$ lifted lines in the common zero locus of $f$ and $g$.
\end{lemma}

Note that when applied to $f$ and $Df$, this gives an upper bound of  
\[\deg_sf(\deg_{x,y}f-1)+\deg_{x,y}f(\deg_sf+1)=2\deg_{x,y}f\deg_sf+\deg_{x,y}f-\deg_sf\]
as $D$ decreases the spatial degree by at least $1$ and increases the slope degree by at most $1$.
If $\deg_{x,y,s}f = \deg_{x,y}f+\deg_sf$, then this bound is better than the previous bound when applied to $f$ and $Df$.

\cref{lemma:mx-improved-bezout} is a direct corollary of a variant of B\'ezout's theorem.
To set up, we will think of $\FF^3$ as $\FF^2\times \FF$ and embed it into $\PP_{\FF}^2\times \PP_{\FF}^1$.
This embedding is defined by $(x,y,s)\mapsto ([x:y:1], [s:1])$.
This allows us to apply tools from intersection theory on $\PP_{\FF}^2\times \PP_{\FF}^1$.
We only state the special case that is relevant to us.

\begin{lemma}[Bihomogeneous B\'ezout's theorem]\label{lemma:bihom-bezout}
    Let $f_1,f_2,f_3\in \FF[x,y,z,s,u]$ be polynomials so that $f_i$ is homogeneous in $x,y,z$ of degree $d_i$ and homogeneous in $s,u$ of degree $e_i$ for each $i\in[3]$.
    Assume that the intersection of the zero loci of $f_1$, $f_2$ and $f_3$ is zero-dimensional.
    Then the number of points, counted with multiplicities, in the intersection is $e_1d_2d_3+d_1e_2d_3+d_1d_2e_3$.
\end{lemma}

We refer the readers to \cite[Chapter 4, Example 4.9]{Sha13} for its general form and more details on intersection theory.
For this exposition, we will assume \cref{lemma:bihom-bezout} without a proof.
For readers not familiar with intersection theory, a helpful example is when each $f_i$ is a product of $d_i$ generic linear functions in $x,y,z$ and $e_i$ generic linear functions in $s,u$.
In this example, the zero locus of $f_i$ is of the form
\[(\ell_1\times \PP_{\FF}^1)\cup\cdots\cup (\ell_{d_i}\times \PP_{\FF}^1)\cup (\PP_{\FF}^2\times \{p_1\})\cup\cdots\cup (\PP_{\FF}^2\times \{p_{e_i}\})\]
where $\ell_1,\ldots, \ell_{d_i}$ are generic lines in $\PP_{\FF}^2$ and $p_1,\ldots, p_{e_i}$ are generic points in $\PP_{\FF}^1$.
It is then easy to see that there are exactly $e_1d_2d_3+d_1e_2d_3+d_1d_2e_3$ common zeroes of $f_1$, $f_2$ and $f_3$.
Miao and Xie provide an elementary proof of \cref{lemma:mx-improved-bezout} without using the bihomogeneous B\'ezout's theorem, and interested readers can see \cite[Step III in the proof of Theorem 2.1]{MX26} for details.

We now see how \cref{lemma:mx-improved-bezout} follows immediately from \cref{lemma:bihom-bezout}.
\begin{proof}
    Bihomogenize $f$ as $\Tilde{f}$ via
    \[\Tilde{f}(x,y,z,s,u) = z^{\deg_{x,y}f}u^{\deg_sf}f\left(\frac{x}{z},\frac{y}{z},\frac{s}{u}\right)\]
    and bihomogenize $g$ as $\Tilde{g}$ as well.
    Suppose that $\Tilde{\ell}=\{(t,at+b,a):t\in \FF\}$ is a line contained in the common zero locus of $f$ and $g$.
    Then this line extends to 
    \[\bar{\Tilde{\ell}}\coloneqq\{([t_0:at_0+bt_1:t_1], [a:1]):[t_0:t_1]\in \PP_{\FF}^1\}\]
    in $\PP_{\FF}^2\times \PP_{\FF}^1$, and it lies in the common zero locus of $\Tilde{f}$ and $\Tilde{g}$.
    For a generic linear function $h$ in $x,y,z$, we thus have that $h$ has exactly one zero on $\bar{\Tilde{\ell}}$.
    Furthermore, as there are only finitely many lifted lines in the common zero locus of $f$ and $g$, we may choose $h$ so that different $\bar{\Tilde{\ell}}$'s intersect the zero locus of $h$ at different points.
    As a consequence, the number of lifted lines in the common zero locus of $f$ and $g$ is upper bounded by the number of intersections of the zero loci of $\Tilde{f}$, $\Tilde{g}$ and $h$ in $\PP^2_{\FF}\times \PP^1_{\FF}$.

    Note that $h$ is bihomogeneous of degree $(1,0)$.
    Therefore by \cref{lemma:bihom-bezout}, the number of intersections may be upper bounded by $\deg_{s,u}\Tilde{f}\deg_{x,y,z}\Tilde{g}+\deg_{x,y,z}\Tilde{f}\deg_{s,u}\Tilde{g}+0 = \deg_sf\deg_{x,y}g+\deg_{x,y}f\deg_sg$.
\end{proof}

\subsection{Dealing with exceptional factors}
The previous subsection works under the assumption that $f$ is coprime to $Df = (\partial_x+s\partial_y)f$.
In this subsection, we will deal with the possible common factors of $f$ and $Df$ and the lifted lines within.

Let $h$ be an irreducible factor of $f$ and $Df$.
To understand the lifted lines in the zero locus of $h$, let $\cL(h)$ be the set of $(a,b)\in \FF^2$ for which $h(t,at+b,a)=0$ for all $t\in\FF$.
We may write $f = hq$ for some polynomial $q$, and we have $Df = hDq+qDh$, showing that $h$ divides $qDh$.
This shows that either $h$ divides $q$ or $h$ divides $Dh$.
The former is easy to deal with: if $h$ divides $q$, then $h^2\mid f$.
However, we may simply replace $f$ with $f/h$, and it will still vanish on all the lines in $\Tilde{\cL_r}$ identically.
More generally, we may assume that $f$ is square-free at the beginning, which means $h$ cannot divide $q$.
Therefore we just need to deal with the case when $h$ divides $Dh$.

Since $\deg_{x,y}Dh\leq \deg_{x,y}h-1$, this forces $Dh\equiv 0$.
Intuitively, this means that $h$ is constant when restricted to any lifted line; this is literally true when the characteristic of $\FF$ is zero, and can be made rigorous in general if $\deg_{x,y}h<\fp$.
This allows us to ``descend'' $h$ to a polynomial on the space of lines whose zero locus is simply $\cL(h)$.
More importantly, we will be able to bound the degree of the polynomial by $\deg_s h$ instead of $\deg h$.
The argument in this paragraph is made rigorous in the following lemma.

\begin{lemma}\label{lemma:mx-descend-poly}
    Let $h\in\FF[x,y,s]$ be a polynomial with $Dh\equiv 0$.
    If $\deg_{x,y}h<\fp$, then $\cL(h)$ is precisely the zero locus of the polynomial $\underline{h}(a,b) \coloneqq h(0,b,a)$.
    Moreover, $\deg \underline{h}= \deg_s h$.
\end{lemma}
\begin{proof}
    Let $g\in \FF[t,b,a]$ be defined by $g(t,b,a)\coloneqq h(t,at+b,a)$.
    It is an easy computation to see that $\partial_tg(t,b,a) = (Dh)(t,at+b,a)\equiv 0$.
    Now we may write $g(t,b,a)$ as 
    \[g(t,b,a) = \sum_{i=0}^{\deg_tg}c_i(a,b)t^i\]
    where each $c_i$ is a bivariate polynomial.
    Since $\partial_tg\equiv 0$, we see that
    \[0\equiv \partial_tg(t,b,a) = \sum_{i=1}^{\deg_tg}ic_i(a,b)t^{i-1}.\]
    Since $\deg_tg\leq \deg_{x,y}h<\fp$, this shows that $c_i(a,b)\equiv 0$ for every $i\geq 1$.
    In other words, $g(t,b,a)$ only depends on $a,b$ but not $t$.
    Therefore
    \[(a,b)\in\cL(h)\Leftrightarrow h(t,at+b,a)=0\;\forall t\in\FF\Leftrightarrow g(t,b,a)=0\;\forall t\in \FF \Leftrightarrow g(0,b,a)=0\Leftrightarrow h(0,b,a)=0.\]
    This shows that $\cL(h)$ is precisely the zero locus of $\underline{h}$.

    For the last claim, since $h(t,at+b,a) = \underline{h}(a,b)$ for every $a,b,t\in \FF$, this shows that we have the polynomial identity $h(x,y,s) = \underline{h}(s,y-sx)$.
    Therefore by measuring the degree in $s$, we have $\deg_sh = \deg \underline{h}$.
\end{proof}
\begin{remark}

    This is the first place in the proof where characteristic shows up, and this is necessary.
    For example, when $\fp<\infty$, the lemma fails for $h(x,y,s) = (y^\fp-sx^\fp)-(y-sx)$.
    For this polynomial, one can verify that $\cL(h) = \FF_{\fp}^2$, which is not contained in the zero locus of any linear function.
\end{remark}

\cref{lemma:mx-descend-poly} allows us to bound the number of rich lines that lift to the zero locus of $h$.
Since $\cL(h)$ lies in the zero locus of $\underline{h}$ in the space of lines, we will use the point-line duality and map the points in $\cP$ to lines.
The rich lines become rich points under the duality, and we just need to bound the number of rich points trapped in the zero locus of $\underline{h}$.
This can be done by a standard application of B\'ezout's theorem as shown below.

\begin{lemma}\label{lemma:mx-exceptional-lines}
    Let $r\geq 2$ be a positive integer.
    Let $\cL$ be a set of lines and let $\cP_r$ be the set of $r$-rich points for $\cL$.
    For any nonzero polynomial $g\in \FF[x,y]$, we have
    \[\abs{\{(x,y)\in \cP_r:g(x,y)=0\}}\leq \frac{\abs{\cL}\deg g}{r-1}.\]
\end{lemma}
\begin{proof}
    Let $\cP$ be the subset of $\cP_r$ consisting of points at which $g$ vanishes.
    By the definition of $\cP_r$, we have $I(\cP,\cL)\geq r\abs{\cP}$.
    To upper bound the size of $\cP$, it thus suffices to upper bound the number of incidences between $\cP$ and $\cL$.

    Any line that is not contained in the zero locus of $g$ intersects it at most $\deg g$ times by B\'ezout's theorem.
    Therefore those lines contribute at most $\abs{\cL}\deg g$ incidences.
    On the other hand, any line $\ell$ contained in the zero locus of $g$ contains at most $\frac{\abs{\cL}-1}{r-1}$ points from $\cP_r$ as each $r$-rich point on $\ell$ is on $r-1$ other lines in $\cL$.
    There are at most $\deg g$ lines in the zero locus of $g$, contributing at most $\abs{\cL}\deg g/(r-1)$ incidences.
    Adding the two bounds gives
    \[r\abs{\cP}\leq I(\cP,\cL)\leq \frac{r\abs{\cL}\deg g}{r-1},\]
    and rearranging gives the desired bound.    
\end{proof}

\subsection{Finishing the proof}
We have now introduced all ingredients in Miao and Xie's proof.
In this subsection, we show how these prove \cref{thm:mx-bound-rich-line}.

\begin{proof}[Proof of \cref{thm:mx-bound-rich-line}]
    We first deal with the main case $r\leq 2\abs{\cP}^{1/2}$.
    Let $\Tilde{\cL_r}$ be the set of lifted lines.
    Let $e = \left\lceil \frac{2\abs{\cP}}{r^2}\right\rceil$, $m=2e$ and $d=rm-1=2re-1$.
    Then
    \[\binom{d+2}{2}(e+1) = \frac{(d+2)(d+1)(e+1)}{2} > \frac{(2re)\cdot (rm) \cdot (e+1)}{2}\geq  2(e+1)m\abs{\cP}\]
    and
    \[\left(e+\frac{m+1}{2}\right)m\abs{\cP} = \left(2e+\frac{1}{2}\right)m\abs{\cP}<2(e+1)m\abs{\cP}.\]
    \cref{lemma:mx-parameter-counting} and \cref{lemma:mx-vanishing-lemma} together give a nonzero polynomial $f\in \FF[x,y,s]$ with $\deg_{x,y} f\leq d$, $\deg_sf\leq e$ and $f$ vanishes on all the lines in $\Tilde{\cL_r}$.
    By removing repeated factors, we may assume that $f$ is square-free.

    Recall that we defined $D=\partial_x+s\partial_y$.
    For any irreducible factor $h$ of $f$, say it is \emph{exceptional} if $Dh\equiv 0$, and it is \emph{normal} otherwise.
    Let $f_{\textup{normal}}$ be the product of normal factors of $f$, and $f_{\textup{ex}}$ be the product of exceptional factors of $f$.
    It is clear that each line in $\Tilde{\cL_r}$ is contained in the union of the zero loci of $f_{\textup{normal}}$ and $f_{\textup{ex}}$.
    By \cref{lemma:mx-improved-bezout} applied to $f_{\textup{normal}}$ and $Df_{\textup{normal}}$, there are at most $(2d-1)\deg_s f_{\textup{normal}}+d$ lines in $\Tilde{\cL_r}$ contained in the zero locus of $f_{\textup{normal}}$.
    It remains to bound the number of $\ell\in \cL_r$ with $\Tilde{\ell}$ lying completely in the zero  locus of $f_{\textup{ex}}$.

    Let $\cL_{\textup{ex}}$ be the set of lines in $\cL_r$ that lift within the zero locus of $f_{\textup{ex}}$.
    Consider the point-line duality $\phi$ that sends lines $\{(t,at+b):t\in\FF\}$ to points $(a,b)\in \FF^2$ and points $p=(p_x,p_y)\in \FF^2$ to lines $\{(a,b)\in \FF^2:p_y=ap_x+b\}$.
    Then it is clear that $\phi(\cL_{\textup{ex}})\subseteq \cL(f_{\textup{ex}})$.
    Note that we assume $r\leq 2\abs{\cP}^{1/2}$ and so $r\leq 4\abs{\cP}/r$.
    The theorem statement also assumes $r\geq 12\abs{\cP}/\mathfrak{p}$.
    Therefore
    \[\deg_{x,y}f_{\textup{ex}}\leq d<2re<\frac{4\abs{\cP}}{r}+2r\leq \frac{12\abs{\cP}}{r}\leq \mathfrak{p}.\]
    \cref{lemma:mx-descend-poly} then shows that $\phi(\cL_{\textup{ex}})$ lies in the zero locus of $\underline{f_{\textup{ex}}}$.
    Moreover, by point-line duality, we know that $\phi(\cL_{\textup{ex}})$ is a set of $r$-rich points for $\phi(\cP)$.
    By \cref{lemma:mx-exceptional-lines}, we thus have
    \[\abs{\cL_{\textup{ex}}}\leq \frac{\abs{\phi(\cP)}\deg\underline{f_{\textup{ex}}}}{r-1} = \frac{\abs{\cP}\deg_s f_{\textup{ex}}}{r-1}\]
    where the last equality holds by \cref{lemma:mx-descend-poly}.

    Adding up, we get
    \[\abs{\cL_r}\leq (2d-1)\deg_sf_{\textup{normal}}+\frac{\abs{\cP}\deg_sf_{\textup{ex}}}{r-1}+d<4re(\deg_sf_{\textup{normal}}+\deg_sf_{\textup{ex}})+2re\leq 4re^2+2re.\]
    Plugging in $e<\frac{2\abs{\cP}}{r^2}+1$ gives
    \[\abs{\cL_r}\leq \frac{16\abs{\cP}^2}{r^3}+\frac{20\abs{\cP}}{r}+6r\leq \frac{16\abs{\cP}^2}{r^3}+\frac{44\abs{\cP}}{r}.\]

    For the remaining case $r>2\abs{\cP}^{1/2}$, the bound follows from a simple combinatorial argument.
    Let $t = \abs{\cL_r}$ and list the lines $\ell_1,\ldots, \ell_t$ in $\cL_r$ in any order.
    Each $\ell_i$ contains at least $r$ points from $\cP$, and at most $(i-1)$ of them are also on one of the lines $\ell_1,\ldots, \ell_{i-1}$.
    Therefore there are at least $r-i+1$ points in $\cP$ that lie in $\ell_i$ and none of the lines $\ell_1,\ldots, \ell_{i-1}$. 
    Summing this over $i$, we get
    \[\sum_{i=1}^{t}\max\left\{r-i+1,0\right\}\leq \abs{\cP}.\]
    Now if $t\geq \left\lceil \frac{2\abs{\cP}}{r}\right\rceil$, then by considering just the first $\left\lceil \frac{2\abs{\cP}}{r}\right\rceil$ summands,
    \[\sum_{i=1}^{t}\max\left\{r-i+1,0\right\}> \left\lceil \frac{2\abs{\cP}}{r}\right\rceil\left(r-\frac{2\abs{\cP}}{r}\right)\geq \frac{2\abs{\cP}}{r}\cdot \frac{r}{2} = \abs{\cP},\]
    which is a contradiction.
    Therefore $\abs{\cL_r}=t<\left\lceil \frac{2\abs{\cP}}{r}\right\rceil$, which shows that $\abs{\cL_r}<\frac{2\abs{\cP}}{r}$.
\end{proof}

Finally, we deduce \cref{thm:st} using \cref{thm:mx-bound-rich-line}.

\begin{proof}[Proof of \cref{thm:st} using \cref{thm:mx-bound-rich-line}]
    With a slight abuse of notation, for any positive integer $r$, let $\cL_r$ be the set of lines in $\cL$ that contain at least $r$ points in $\cP$.
    Let $T_1, T_2$ be two positive integers.
    It is clear that
    \[I(\cP,\cL)\leq \sum_{r=1}^{T_2-1}\abs{\cL_r} + I(\cP,\cL_{T_2})\]
    as any line in $\cL$ containing exactly $i$ points is either counted $i$ times in the first sum or is included in $\cL_{T_2}$.
    By bounding $\abs{\cL_r}\leq \abs{\cL}$ for all $r\leq T_1$, we get
    \begin{equation}\label{eq:mx}
        I(\cP,\cL) \leq T_1\abs{\cL}+\sum_{r=T_1+1}^{T_2-1}\abs{\cL_r}+I(\cP,\cL_{T_2})
    \end{equation}
    where the middle sum is the empty sum if $T_1+1>T_2-1$.

    Set $T_2 = \lceil2\abs{\cP}^{1/2}\rceil$.
    By the proof of \cref{thm:mx-bound-rich-line}, we know that $\abs{\cL_{T_2}}\lesssim \frac{\abs{\cP}}{T_2}\lesssim \abs{\cP}^{1/2}$.
    Hence
    \begin{align*}
        I(\cP,\cL_{T_2})\leq &\sum_{p\in \cP}\abs{\left\{\ell\in \cL_{T_2}:\ell\ni p\right\}} \leq \sum_{p\in \cP}\left(1+\abs{\left\{\{\ell,\ell'\}\in \binom{\cL_{T_2}}{2}:p\in \ell,\ell'\right\}}\right)\\\leq& \abs{\cP}+\binom{\abs{\cL_{T_2}}}{2}\lesssim \abs{\cP}
    \end{align*}
    as any two lines intersect at most once.

    It remains to bound the first two terms of \cref{eq:mx}.
    To apply \cref{thm:mx-bound-rich-line}, we will take $T_1\geq \max\{2,\frac{12\abs{\cP}}{\fp}\}$.
    Assuming this is the case, we have
    \[\sum_{r=T_1+1}^{T_2-1}\abs{\cL_r}\lesssim \sum_{r=T_1+1}^{T_2-1}\left(\frac{\abs{\cP}^2}{r^3}+\frac{\abs{\cP}}{r}\right)\lesssim \sum_{r=T_1+1}^{T_2-1} \frac{\abs{\cP}^2}{r^3}\lesssim \frac{\abs{\cP}^2}{T_1^2}\]
    where we use that $\frac{\abs{\cP}}{r}\lesssim \frac{\abs{\cP}^2}{r^3}$ when $r<T_2$.
    Hence
    \[I(\cP,\cL)\lesssim T_1\abs{\cL}+\frac{\abs{\cP}^2}{T_1^2}+\abs{\cP}\]
    for any $T_1\geq \max\{2,\frac{12\abs{\cP}}{\fp}\}$.
    We will take $T_1$ to be the least positive integer satisfying this inequality and also $\frac{\abs{\cP}^2}{T_1^2}\leq T_1\abs{\cL}$.
    Then we have
    \[T_1\lesssim 1+\frac{\abs{\cP}}{\fp}+\frac{\abs{\cP}^{2/3}}{\abs{\cL}^{1/3}},\]
    and so
    \[I(\cP,\cL)\lesssim T_1\abs{\cL}+\abs{\cP}\lesssim \abs{\cP}+\abs{\cL}+\abs{\cP}^{2/3}\abs{\cL}^{2/3}+\frac{\abs{\cP}{\abs{\cL}}}{\fp}.\qedhere\]
\end{proof}
\begin{remark}
    Optimizing the constant for $\abs{\cP}^{2/3}\abs{\cL}^{2/3}$ using the above argument, one can get arbitrarily close to $3\sqrt[3]{2}$ at the cost of arbitrarily large constants for other terms.
    This is worse than the constant $3$ produced by Lewko's proof.
    
    On the other hand, Lewko's bound shows that the number of $r$-rich lines can be bounded as follows:
    \[r\abs{\cL_r}\leq I(\cP,\cL_r)\leq \abs{\cP}+\abs{\cL_r}+3\abs{\cP}^{2/3}\abs{\cL_r}^{2/3}+\frac{2\abs{\cP}\abs{\cL_r}}{\fp},\]
    which only gives roughly an upper bound of $\frac{27\abs{\cP}^2}{r^3}+\frac{2\abs{\cP}}{r}$ even when $r$ is much larger than $\max\{2,\frac{2\abs{\cP}}{\fp}\}.$
    Therefore the two proofs give bounds of incomparable strengths.
\end{remark}

\section{Lewko's proof}\label{sec:lewko}
We present Lewko's proof of \cref{thm:st} in this section.
As mentioned in the introduction, \cref{thm:st} follows immediately from \cref{thm:lewko-st-with-param}.
We first prove this implication.
\begin{proof}[Proof of \cref{thm:st} assuming \cref{thm:lewko-st-with-param}]
    If $\abs{\cP}\abs{\cL}<\fp^3$, taking $d=\abs{\cP}^{1/3}\abs{\cL}^{1/3}$ in \cref{thm:lewko-st-with-param} gives
    \[I(\cP,\cL)\leq \abs{\cP}+\abs{\cL}+3\abs{\cP}^{2/3}\abs{\cL}^{2/3}.\]
    Otherwise, take $d$ arbitrarily close to $\fp$.
    Then we get
    \[I(\cP,\cL)\leq \abs{\cP}+\abs{\cL}+\fp^2+\frac{2\abs{\cP}\abs{\cL}}{\fp}\leq \abs{\cP}+\abs{\cL}+\abs{\cP}^{2/3}\abs{\cL}^{2/3}+\frac{2\abs{\cP}\abs{\cL}}{\fp}.\]
    Therefore in any case, we have
    \[I(\cP,\cL)\leq \abs{\cP}+\abs{\cL}+3\abs{\cP}^{2/3}\abs{\cL}^{2/3}+\frac{2\abs{\cP}\abs{\cL}}{\fp}.\qedhere \]
\end{proof}
Therefore we will focus on proving \cref{thm:lewko-st-with-param} for the rest of the section.
\subsection{Parameter counting}
Let $\cP$ be a set of points and $\cL$ be a set of lines in $\FF^2$.
Let $\Tilde{\cL}$ be the set of lifted lines.
As in Miao and Xie's proof, we begin by constructing a polynomial of small degree that vanishes identically on every line in $\Tilde{\cL}$.
In Lewko's proof, this is done by a simple parameter counting argument.

\begin{lemma}[Parameter counting]\label{lemma:param-counting}
    Let $\Tilde{\cL}$ be the set of lifted lines in $\FF^3$.
    Let $d,e$ be any non-negative integers such that $(d+2)(e+1)>2\abs{\Tilde{\cL}}$.
    Then there exists a nonzero polynomial $f\in \FF[x,y,s]$ vanishing identically on all lines in $\Tilde{\cL}$ with $\deg_{x,y}f\leq d$ and $\deg_s f\leq e$.
\end{lemma}
\begin{proof}
    The space of polynomials with spatial degree at most $d$ and slope degree at most $e$ has dimension exactly
    \[\binom{d+2}{2}(e+1).\]
    Asking it to vanish identically on a line in $\Tilde{\cL}$ within any given slice imposes $d+1$ linear constraints.
    Indeed, any line in $\Tilde{\cL}$ is of the form $\{(t,at+b,a):t\in \FF\}$, and so $f(t,at+b,a)$ can be written as a polynomial of degree at most $d$.
    Asking $f$ to vanish on the line is the same as asking the $d+1$ coefficients to be zero.
    Therefore there is a nonzero polynomial $f$ in the space vanishing on all lines if
    \[\binom{d+2}{2}(e+1) > (d+1)\abs{\cL}.\]
    By the choice of the parameters $(d,e)$, we know this inequality holds.
\end{proof}

\subsection{From incidence to multiplicity}
Suppose we have a polynomial $f$ produced from \cref{lemma:param-counting}.
It is then clear that if $p$ is an $r$-rich point and $r>e\geq \deg_sf$, then $f$ must vanish identically on $\ell_p$.
Indeed, $\ell_p$ has $r$ distinct intersections with lines from $\Tilde{\cL}$ and thus contains $r$ points lying in the zero locus of $f$.
Since $\deg_s f\leq e<r$, this forces the entire line $\ell_p$ to lie in the zero locus of $f$.
It turns out that if $r$ is much larger than $e$, we can prove a much stronger statement.

Recall that $\mu_f(p,s)$ is the multiplicity of $f$ at $(p,s)$ along the line $\ell(p,s)$.
If $p$ is $r$-rich, then $\mu_f(p,s)=\infty$ holds for at least $r$ different values of $s$.
The following lemma uses this to show that $\mu_f(p,s)$ has to be large for any $s$.
In some sense, this is the converse of the vanishing lemma (\cref{lemma:mx-vanishing-lemma}) in Miao and Xie's proof.

\begin{lemma}[Incidence to multiplicity]\label{lemma:incidence-to-mult}
    Let $e$ be any nonnegative integer.
    Let $p$ be a point lying on $r$ lines in $\cL$ in the plane $\FF^2$, and let $\Tilde{\cL}$ be the set of lifted lines.
    Suppose that $f\in \FF[x,y,s]$ is a polynomial vanishing on all the lines in $\Tilde{\cL}$ with $\deg_sf\leq e$.
    Then $\mu_f(p,s)\geq r-e$ for any $s\in\FF$.
\end{lemma}
\begin{proof}
    Let $p = (p_x,p_y)$.
    Write $g(s,t) = f(p_x+t, p_y+st, s)$, and treat $g$ as a polynomial in $(\FF[s])[t]$.
    That is, write $g$ as 
    \[g(s,t) = \sum_{i=0}^{\deg_t g}c_i(s)t^i\]
    where each $c_i\in \FF[s]$.
    Now fix some $i=0,\ldots, \deg_tg$.
    By our assumptions, $g(s,t)$ is a zero polynomial in $t$ for $r$ different values of $s$, which implies that $c_i(s)$ has $r$ roots.
    We would like to show that this forces $c_i(s)\equiv 0$ when $i$ is small.
    To do this, we need to bound the degree of $c_i$.
    This was already done in the proof of \cref{lemma:mx-parameter-counting}, but we state it again to keep the proof independent.

    The key claim is that $\deg c_i \leq e+i$.
    To see this, notice that $g(s,t)$ is an $\FF$-linear combination of 
    \[(p_x+t)^{\alpha}(p_y+st)^{\beta}s^{\gamma}\]
    where $\gamma\leq e$.
    Notice that after expanding the first two factors, we always have that the degree of $s$ is at most the degree of $t$.
    Multiplying by $s^{\gamma}$ only increases the degree by at most $e$.
    Therefore $\deg c_i\leq e+i$.

    Now this immediately shows that $c_i(s)\equiv 0$ for every $i$ with $e+i<r$.
    Rearranging shows that $c_i(s)\equiv 0$ for every $i<r-e$, which shows that $\mu_f(p,s)\geq r-e$ for any $s\in \FF$ by definition.
\end{proof}

\begin{remark}
    This lemma is arguably the most magical step in Lewko's proof, and the author has little geometric intuition for why such a lemma could hold.
    Perhaps for some intuition, when $e=r-1$ the lemma shows that $\mu_f(p,s)\geq 1$ for all $s$, which is the same as saying $f$ vanishes identically on $\ell_p$.
    At the other extreme, when $e=0$, the polynomial $f$ only depends on the spatial component, and the lemma says that $\mu_f(p,s)\geq r$.
    This is the same as saying any polynomial vanishing on $r$ distinct lines through the origin must have multiplicity at least $r$ at the origin (in all directions).
\end{remark}

\subsection{Bounding the sum of multiplicities}

\cref{lemma:incidence-to-mult} allows us to bound incidences with some multiplicities computed in any slice, with a cost of $e$ incidences per point.
Since $e$ can be chosen to be at most $2\abs{\cL}/d$, we only lose $\frac{2\abs{\cP}\abs{\cL}}{d}$ going from incidences to multiplicities.
Therefore the goal is to bound $\sum_{p\in \cP}\mu_f(p,s)$ by, say, $d(d-1)$.
As the lower bound in \cref{lemma:incidence-to-mult} holds for any $s\in \FF$, we will pick a generic $s$ and then zoom into the $s$-th slice.

Fix some choice of $s$.
View $f_s(x,y)\coloneqq f(x,y,s)$ as a polynomial in $\FF[x,y]$.
Then $\mu_f(p,s)$ is simply the multiplicity of $f_s(x,y)$ at $p$ along the line $\{(p_x+t,p_y+st):t\in \FF\}$.
Notice that by a linear change of coordinates in $x$ and $y$, we can make those lines horizontal.
Therefore our goal is to bound the sum of multiplicities of $f_s$ at all the $p\in \cP$ along the horizontal direction.

\begin{figure}[h]
    \centering
    \begin{tikzpicture}[
    x=4cm, y=1.8cm,
    line cap=round, line join=round,
    font=\small,
    curve/.style={blue, very thick},
    horizontal/.style={gray!75, dashed, thick}
]

    \draw[horizontal]
        (-0.78,{11/16}) -- (1.58,{11/16});

    \draw[horizontal]
        (-0.78,0) -- (1.58,0);

    \draw[horizontal]
        (-0.78,-1) -- (1.58,-1);

    \draw[
        curve,
        domain=-0.60:1.45,
        samples=240,
        variable=\t
    ]
        plot (\t,{3*(\t)^4 - 4*(\t)^3});

    \node[blue,above left]
        at (1.45,{3*(1.45)^4 - 4*(1.45)^3}) {$f_s(x,y)=y-(3x^4-4x^3)$};

    \coordinate (P1) at (-0.5,{11/16});
    \coordinate (P3) at (0,0);
    \coordinate (P2) at (1,-1);

    \foreach \p in {P1,P3,P2}
        \fill (\p) circle[radius=2pt];

    \node[above right,xshift=2pt,yshift=2pt] at (P1)
        {$p_1$ with $\mu_f(p_1,s)=1$};

    \node[above right,xshift=2pt,yshift=2pt] at (P3)
        {$p_3$ with $\mu_f(p_3,s)=3$};

    \node[below,yshift=-3pt] at (P2)
        {$p_2$ with $\mu_f(p_2,s)=2$};

\end{tikzpicture}
    \caption{An example with three points $p_1,p_2,p_3$ and their corresponding multiplicities.}
    \label{fig:curve}
\end{figure}
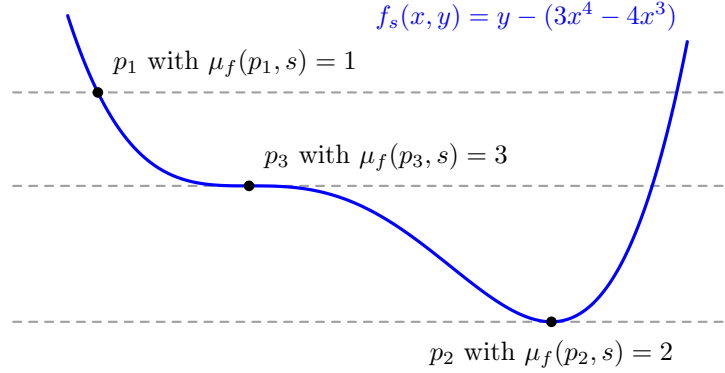

It is natural to expect that although there may be infinitely many points with multiplicity $1$, there are only finitely many of them with $2$ and above.
The strategy is therefore to bound the sum of multiplicities for those finitely many ``critical'' points.
However, two technicalities must be dealt with.
First, if $f_s$ has a horizontal linear component, then there are infinitely many points with infinite multiplicity.
Even without horizontal linear components, $f_s$ might not be square-free, resulting in infinitely many points with multiplicities at least $2$.
Barring those two examples, one can indeed bound the sum of multiplicities given $\deg f_s<\fp$, as shown in the lemma below.

Before we state the lemma, we introduce the following notation.
For any $g\in \FF[x,y]$ and $p\in \FF^2$, let $\mu_g^{\textup{hor}}(p)$ be the multiplicity of $g$ at $p$ along the horizontal direction.
In other words, it is the largest non-negative integer $m$ so that $t^m$ divides $g(p_x+t,p_y)$ as polynomials in $t$.
If $g(p_x+t,p_y)$ is the zero polynomial in $t$, set $\mu_g^{\textup{hor}}(p)=\infty$.
\begin{lemma}\label{lemma:bound-sum-mult}
    Let $g\in \FF[x,y]$ be a square-free polynomial of degree $d<\fp$.
    Suppose that there are no factors of $g$ that only depend on $y$.
    Then
    \[\sum_{p\in \FF^2}\max\left\{\mu_g^{\textup{hor}}(p)-1,0\right\}\leq d(d-1).\]
\end{lemma}

We first remark that the technical-looking assumption $\deg g<\fp$ is necessary.
For example, one can easily check that any zero of the polynomial $g(x,y) = x^{\fp}-y$ has horizontal multiplicity exactly $\fp$, and so the above sum becomes infinite.
This is why we need $d<\fp$ in \cref{thm:lewko-st-with-param}.

We also remark that this lemma follows from a fairly standard application of B\'ezout's theorem, and similar results may have been recorded in the literature especially for the characteristic zero case.
In \cite{Lew26}, the author deliberately avoided mentioning B\'ezout's theorem and replaced it with a parameter counting argument.

To prove \cref{lemma:bound-sum-mult}, the example in \cref{fig:curve} provides a good hint.
In that case, the curve is the graph of the function $y=3x^4-4x^3$, and points with multiplicities more than $1$ are exactly the critical points of the polynomial $3x^4-4x^3$.
They are thus roots of the derivative of $3x^4-4x^3$---moreover, the multiplicity of the derivative along the horizontal direction is exactly one less than the multiplicity of $3x^4-4x^3$ along the horizontal direction.
This strategy works in general so long as the ($x$-directional) derivative is coprime to the original polynomial.

\begin{proof}
    We first show that $\partial_x g$ is coprime to $g$.
    Suppose for the sake of contradiction that $h$ is an irreducible common factor between $\partial_x g$ and $g$.
    Then we may write $g=hq$ for some polynomial $q$, and $\partial_x g = h\partial_xq + q\partial_xh$.
    Since $g$ is square-free, this forces $h$ to share a common factor with $\partial_xh$.
    As $\partial_x h$ has strictly smaller degree, this forces $\partial_x h\equiv 0$.

    By our assumption on $g$, $\deg_x h\geq 1$.
    Moreover, we know $\deg_x h \leq \deg_x g < \fp$.
    Therefore $h$ must contain some term $cx^{\alpha}y^{\beta}$ with $1\leq \alpha < \fp$ and $c\neq 0$.
    Now $\partial_x h$ contains the term $c\alpha x^{\alpha-1}y^{\beta}$, which is nonzero.
    This is a contradiction.

    We now claim that the intersection multiplicity of $\partial_xg$ and $g$ at any point $p\in \FF^2$ is at least $\mu^{\textup{hor}}_g(p)-1$.
    Geometrically, this is because both the zero locus of $g$ and the zero locus of $\partial_x g$ are tangent to the horizontal line through $p$ to order $\mu^{\textup{hor}}_g(p)-1$.
    The claim for $\partial_xg$ can be verified easily after expanding the definition.
    Algebraically, the intersection multiplicity may be computed as
    \[\dim_{\FF}\cO_p/\langle g,\partial_xg\rangle\]
    where $\cO_p$ consists of rational functions $u/v$ where $v(p)\neq 0$.
    Note that when $g$ is expanded around $p$, it does not contain monomials of the form $(x-p_x)^i(y-p_y)^0$ with $i<\mu_g^{\textup{hor}}(p)$.
    Therefore when $\partial_xg$ is expanded around $p$, it does not contain monomials of the form $(x-p_x)^i(y-p_y)^0$ with $i<\mu_g^{\textup{hor}}(p)-1$.
    This shows that $(x-p_x)^0,\ldots, (x-p_x)^{\mu_g^{\textup{hor}}(p)-2}$ are linearly independent in the quotient, proving the claim.

    Now by B\'ezout's theorem, the intersection of the zero loci of $g$ and $\partial_xg$ contains at most $d(d-1)$ points counted with multiplicities.
    In other words,    
    \[\sum_{p\in \FF^2}\max\left\{\mu_g^{\textup{hor}}(p)-1,0\right\}\leq d(d-1).\qedhere\]
\end{proof}

\subsection{Choosing a nice slice}
In this subsection, we will show that we can choose a slice $s$ so that $f_s$ is nice in some sense.
Before we begin, we remark that work done in this subsection can be mostly circumvented by working over the field $\FF(s)$ where $s$ is now treated as a formal variable.
In fact, we will need to work over $\FF(s)$ in \cref{sec:main} when we modify Lewko's proof.
However, to make the argument more intuitive geometrically, we choose to proceed here by fixing a generic value for $s$ and showing that properties we need for $f_s$ still hold.

To apply \cref{lemma:bound-sum-mult} to bound the sum of $\mu_f(p,s)$, we would like $f_s$ to be square-free and contain no linear components in the $(1,s)$-direction.
The latter creates a bigger issue, which we deal with first.

The linear components of $f_s$ that do not contain any points from $\cP$ do not affect the multiplicities.
The real issue is when there is a linear component in the $(1,s)$-direction containing some point from $\cP$.
This is equivalent to $f$ vanishing identically on $\ell(p,s)$.
From this characterization, we see that this indeed happens for some values of $s$---indeed, any $s$ that appears as a slope for some line in $\cL$ passing through at least one point in $\cP$ would create some issues.
However, there are finitely many of them, and we can avoid them as we assume $\FF$ is infinite.

The more serious issue is that there might be $p\in \cP$ with $f$ vanishing identically on $\ell(p,s)$ for all $s\in \FF$.
In other words, $f$ vanishes on the regulus $R_p$ swept out by the lines $\{\ell(p,s):s\in \FF\}$.
This will forbid us from choosing any good $s$.
To deal with those exceptional points, let $\cP_{\textup{ex}}$ be the set of points $p\in \cP$ so that $f$ vanishes identically on $R_p$.
We will show that $\cP_{\textup{ex}}$ is small, which will allow us to deal with incidences involving them using some easy bounds.

\begin{lemma}\label{lemma:small-exceptional}
    Let $f\in \FF[x,y,s]$ be a nonzero polynomial. Let $\cP_{\textup{ex}}$ be a set of points in $\FF^2$ for which $f$ vanishes identically on $R_p$.
    Then $\abs{\cP_{\textup{ex}}}\leq \deg_sf$.
\end{lemma}
The proof of the lemma is simple, but we would like to emphasize a subtle point.
If one just considers the total degree, then one may get $\abs{\cP_{\textup{ex}}}\leq \frac{1}{2}\deg f$, which might be significantly larger than $\deg_sf$ if $\deg_{x,y}f$ is much larger than $\deg_sf$.
Instead, we will only measure degrees in the slope component.

\begin{proof}
    We first note that the defining polynomial of $R_p$ for some $p=(p_x,p_y)\in \cP$ is $(y-p_y)-s(x-p_x)$, which is irreducible and has degree one in $s$.
    Moreover, distinct points give distinct defining polynomials.
    Therefore $f$ is divisible by 
    \[\prod_{p\in \cP_{\textup{ex}}}\left((y-p_y)-s(x-p_x)\right),\]
    which shows that $\abs{\cP_{\textup{ex}}}\leq \deg_sf$.    
\end{proof}

Except for points from $\cP_{\textup{ex}}$, we can now show that $f$ does not vanish identically on $\ell(p,s)$ for all but finitely many $s\in \FF$.

\begin{lemma}\label{lemma:good-slice-line}
    For any $p\in \cP\backslash \cP_{\textup{ex}}$, there are at most $2\deg f$ values of $s$ for which $f$ vanishes identically on $\ell(p,s)$.
\end{lemma}
\begin{proof}
    For any $s$ with $f$ vanishing identically on $\ell(p,s)$, we know that the line $\ell(p,s)$ lies in both $R_p$ and the vanishing locus of $f$.
    Since $f$ does not vanish identically on $R_p$, by B\'ezout's theorem, the vanishing locus of $f$ cuts $R_p$ into a $1$-dimensional algebraic set of degree at most $2\deg f$.
    In particular, there are at most $2\deg f$ linear components in the $1$-dimensional algebraic set.
    This shows that there can be at most $2\deg f$ values of $s$ for which $f$ vanishes identically on $\ell(p,s)$.
\end{proof}

Finally, we will show that $f_s$ is square-free for most $s\in \FF$ provided that $f$ is square-free.
The argument is similar to the one in \cref{lemma:bound-sum-mult} where we show that $g$ and $\partial_xg$ have no common factors.
\begin{lemma}\label{lemma:good-slice-square-free}
    Let $f\in \FF[x,y,s]$ be a square-free polynomial with $\deg_{x,y}f<\fp$.
    Then except for at most $(\deg f)^2$ values of $s$, the polynomial $f_s\in\FF[x,y]$ is square-free.
\end{lemma}
\begin{proof}
    Factor $f = g_1g_2$, where $g_1$ only depends on $s$ and $g_2$ has no factors that only depend on $s$.
    We first show that $g_2,\partial_xg_2,\partial_yg_2$ have no common factors.
    Otherwise, let $h$ be an irreducible common factor.
    As before, it is then clear that $h,\partial_xh,\partial_yh$ must share a common factor as $g_2$ is square-free, showing that $h\mid \partial_xh,\partial_yh$.
    This forces $\partial_xh\equiv \partial_yh\equiv 0$.
    However, as $h$ does not depend only on $s$, it must contain a term $cx^{\alpha}y^{\beta}s^{\gamma}$ with $c\neq 0$, $(\alpha,\beta)\neq (0,0)$ and $\alpha+\beta<\fp$.
    This forces one of $\partial_xh,\partial_yh$ to be nonzero, which is a contradiction.
    Therefore $g_2,\partial_xg_2,\partial_yg_2$ do not share common factors.

    B\'ezout's theorem then shows that the common zero locus $Z$ of $g_2,\partial_xg_2,\partial_yg_2$ has $1$-dimensional components of total degree at most $(\deg g_2)^2$.
    In particular, there are at most $(\deg g_2)^2$ slices that contain any such component entirely.
    Now we will show that if $g_1(s)\neq 0$ and none of the $1$-dimensional components of $Z$ lies entirely in the $s$-th slice, then $f_s$ is square-free.
    Once this is proven, we just need to upper bound
    \[\deg g_1+(\deg g_2)^2\leq (\deg g_1+\deg g_2)^2 =(\deg f)^2.\]
    Suppose now that $s$ is such that $g_1(s)\neq 0$ and none of the $1$-dimensional components of $Z$ is contained in the $s$-th slice.
    Then $f_s,\partial_xf_s,\partial_yf_s$ are nonzero constant multiples of $g_2,\partial_xg_2,\partial_yg_2$ within the slice, respectively, and there are only finitely many common zeroes by our assumption on $s$.
    As $\FF$ is algebraically closed, this means that $f_s,\partial_xf_s$ and $\partial_yf_s$ have no common factors.
    This shows that $f_s$ is square-free.
\end{proof}
\begin{remark}
    The condition $\deg_{x,y}f<\fp$ is again necessary, as shown by the polynomial $f = x^{\fp}-s$.
    As mentioned earlier, this technical lemma can be avoided entirely if we work over the field $\FF(s)$.
\end{remark}
\subsection{Finishing the proof}
We can now put every ingredient together to prove \cref{thm:lewko-st-with-param}.
\begin{proof}[Proof of \cref{thm:lewko-st-with-param}]
Let $e = \lfloor 2\abs{\cL}/d\rfloor$.
It is clear that $(\lfloor d\rfloor +2)(e+1)>d(e+1)>2\abs{\cL}$.
Lift $\cL$ up to $\Tilde{\cL}$ in $\FF^3$, and let $f$ be a polynomial given by \cref{lemma:param-counting} with parameters $(\lfloor d\rfloor,e)$.
We may assume that $f$ is square-free by removing any repeated factors.
Let $\cP_{\textup{ex}}$ be the set of points in $\cP$ for which $f$ vanishes identically on $R_p$.
For any $p\in \cP$, let $r(p)$ be the number of lines in $\cL$ through $p$.

We first bound $I(\cP\backslash \cP_{\textup{ex}},\cL)$.
By \cref{lemma:incidence-to-mult}, 
\[I(\cP\backslash \cP_{\textup{ex}},\cL)=\sum_{p\in \cP\backslash \cP_{\textup{ex}}}r(p)\leq (e+1)\left(\abs{\cP}-\abs{\cP_{\textup{ex}}}\right)+\sum_{p\in \cP\backslash \cP_{\textup{ex}}:\mu_f(p,s)>1}(\mu_f(p,s)-1)\]
for any $s\in \FF$. 
We now pick a good slice to bound the last sum.
By \cref{lemma:good-slice-line} and \cref{lemma:good-slice-square-free}, we may pick $s\in \FF$ so that $f_s$ is square-free and $f$ does not vanish identically on $\ell(p,s)$ for any $p\in\cP\backslash \cP_{\textup{ex}}$.
Fix this choice of $s$ for the rest of the argument.

By a change of coordinates in $x,y$, we may assume that $s=0$.
Then $\mu_f(p,s) = \mu^{\textup{hor}}_{f_s}(p)$ for any $p\in \cP\backslash \cP_{\textup{ex}}$.
Remove all horizontal linear components of $f_s$ to get $g$.
It is clear that $\mu_{g}^{\textup{hor}}(p) = \mu_{f_s}^{\textup{hor}}(p)$ for any $p\in \cP\backslash \cP_{\textup{ex}}$ as none of the horizontal lines contain any points from $\cP\backslash \cP_{\textup{ex}}$.
Therefore
\[\sum_{p\in \cP\backslash \cP_{\textup{ex}}:\mu_f(p,s)>1}(\mu_f(p,s)-1) = \sum_{p\in \cP\backslash \cP_{\textup{ex}}:\mu^{\textup{hor}}_{g}(p)>1}(\mu^{\textup{hor}}_{g}(p)-1)\leq d(d-1)\]
by \cref{lemma:bound-sum-mult} as $\deg g\leq \deg f_s\leq \deg_{x,y}f\leq d$.
This gives
\[I(\cP\backslash \cP_{\textup{ex}},\cL)\leq (e+1)\left(\abs{\cP}-\abs{\cP_{\textup{ex}}}\right)+d(d-1).\]

It remains to deal with the incidences between $\cP_{\textup{ex}}$ and $\cL$.
Here we use the simple combinatorial bound
\[I(\cP_{\textup{ex}},\cL)\leq \abs{\cL}+\binom{\abs{\cP_{\textup{ex}}}}{2}.\]
To prove it, simply notice that
\begin{align*}
    I(\cP_{\textup{ex}},\cL) =& \sum_{\ell\in \cL}\abs{\ell\cap \cP_{\textup{ex}}}\leq \sum_{\ell\in\cL}\left(1+\binom{\abs{\ell\cap \cP_{\textup{ex}}}}{2}\right)\\ =& \abs{\cL}+\sum_{\ell\in \cL}\abs{\left\{\{p,p'\}:p\neq p'\in \ell\cap \cP_{\textup{ex}}\right\}}\\
    \leq & \abs{\cL}+\binom{\abs{\cP_{\textup{ex}}}}{2}
\end{align*}
as two points determine a unique line. By \cref{lemma:small-exceptional}, we know that $\abs{\cP_{\textup{ex}}}\leq e$.
Therefore we can upper bound $\binom{\abs{\cP_{\textup{ex}}}}{2}$ loosely by $(e+1)\abs{\cP_{\textup{ex}}}$.

Combining the two bounds,
\[I(\cP,\cL)\leq I(\cP\backslash \cP_{\textup{ex}},\cL)+I(\cP_{\textup{ex}},\cL)\leq (e+1)\abs{\cP}+\abs{\cL}+d(d-1).\]
Since $e\leq 2\abs{\cL}/d$, this gives the desired inequality.
\end{proof}

\section{Improved Szemer\'edi--Trotter over arbitrary fields}\label{sec:main}
In this section, we show how Hasse derivatives can be used in both Lewko's and Miao and Xie's proofs to get some improvement.
Since Lewko's proof is the one that will lead to \cref{thm:main}, we will only show the full details of the modification of Lewko's proof and only mention the modification for Miao and Xie's proof briefly.
We also show the tightness of the bound in all regimes when $\FF$ is algebraically closed.

\subsection{Hasse derivative}
We begin by introducing Hasse derivatives and their properties.
For our purpose, we only need Hasse derivatives with respect to a single variable.
As we need to work over other fields later, we let $K$ be any field having the same characteristic as $\FF$ in this subsection.
\begin{definition}[Hasse derivative]
    On the ring of polynomials $K[x_1,\ldots, x_d]$, for any $n\in \ZZ_{\geq 0}$, the \emph{$n$-th Hasse derivative in $x_i$}, denoted by $\Hasse^n_{x_i}$, is a linear operator on $K[x_1,\ldots, x_d]$ so that
    \[f(x_1,\ldots, x_i+t,\ldots, x_d) = \sum_{n=0}^{\infty}(\Hasse^n_{x_i}f)(x_1,\ldots, x_d)t^n\]
    for any $f\in K[x_1,\ldots, x_d]$. In other words, $\Hasse^n_{x_i}f$ is the coefficient of $t^n$ in the expansion of $f$ when $x_i$ is replaced with $x_i+t$.
    We omit the index $x_i$ if it is clear from the context.
\end{definition}
It is easy to verify that $\Hasse^1_{x_i}$ is the same as the formal derivative $\partial_{x_i}$.
Indeed, by a direct expansion, we see that
\[\Hasse^1_{x_1} x_1^{\alpha_1}\cdots x_d^{\alpha_d} = [t^1](x_1+t)^{\alpha_1}x_2^{\alpha_2}\cdots x_d^{\alpha_d} = \alpha_1x_1^{\alpha_1-1}x_2^{\alpha_2}\cdots x_d^{\alpha_d} = \partial_{x_1}\left(x_1^{\alpha_1}\cdots x_d^{\alpha_d}\right),\]
showing that $\Hasse^1_{x_1}$ agrees with $\partial_{x_1}$ on all monomials.
As any polynomial is a linear combination of finitely many monomials, the two derivative operators agree on the entire polynomial ring.
In fact, the same expansion gives that $\partial^n_{x_i} = n!\Hasse^n_{x_i}$, whose proof is omitted as it will not be used.
This shows that over any field of characteristic zero, the higher-order Hasse derivatives can be thought of as partial derivatives of higher order up to a constant.
However, over a field of positive characteristic, we necessarily have $\partial^n_{x_i}f\equiv 0$ for any $f\in K[x_1,\ldots, x_d]$ whenever $n$ is at least the characteristic of the field.
On the other hand, $\Hasse^n_{x_i}$ need not be zero even if $n$ is larger than the characteristic of the field.
This will allow us to get rid of the constraint on the spatial degrees in both proofs at the expense of some additional loss.

\begin{example}\label{ex:why-it-breaks}
    Assume that $\fp<\infty$ for the discussion here.
    In \cref{lemma:mx-descend-poly}, the constraint $\deg_{x,y}h<\fp$ is needed so that when we take $g\in K[t,b,a]$ with $g(t,b,a) = h(t,at+b,a)$, we have that $\partial_tg\equiv 0$ implies $g$ does not depend on $t$.
    This no longer holds when $\deg_{x,y}h\geq \fp$, and as shown in the remark, $h(x,y,s) = (y^\fp-sx^{\fp})-(y-sx)$ shows that the statement can fail in this regime.
    For this choice of $h$, we have $g(t,b,a) = (at+b)^{\fp}-at^{\fp}-b = (a^{\fp}-a)t^p+(b^{\fp}-b)$.
    Note that $g$ does depend on $t$ even with $\partial_tg\equiv 0$.

    However, if we consider the $\fp$-th order Hasse derivative, then we have
    \[\Hasse^{\fp}_tg(t,b,a) = [(t')^{\fp}] (a^{\fp}-a)(t+t')^{\fp}+(b^{\fp}-b) = a^{\fp}-a,\]
    which is not a zero polynomial.
    Therefore, instead of just the formal derivative, we will consider various higher-order Hasse derivatives of $g$ with respect to $t$ to ensure that $g$ does not depend on $t$.

    \cref{lemma:bound-sum-mult} also needs the assumption that $\deg g<\fp$ for the same reason, and we will also fix it by considering higher-order Hasse derivatives.
\end{example}

To prove a generalization of \cref{lemma:bound-sum-mult} using Hasse derivatives, we need to understand how the multiplicity along the horizontal direction behaves.
We first reinterpret the multiplicity using Hasse derivatives.
\begin{proposition}\label{prop:Hasse-is-multi}
    Let $g\in K[x,y]$ be any polynomial and let $p\in K^2$ be any point.
    Then $\mu^{\textup{hor}}_{g}(p)$ is the least nonnegative integer $m$ so that
    \[\Hasse^m_xg(p)\neq 0.\]
    If such $m$ does not exist, then  $\mu^{\textup{hor}}_{g}(p)=\infty$.
\end{proposition}
\begin{proof}
    Recall from the definition of Hasse derivative that
    \[g(p_x+t,p_y) = \sum_{n=0}^{\infty}\Hasse^n_{x}g(p)t^n.\]
    By the definition of $\mu^{\textup{hor}}_g(p)$, it is the largest nonnegative integer $m$ so that $t^m$ divides $g(p_x+t,p_y)$ as polynomials in $t$.
    Since $t^m\mid g(p_x+t,p_y)$ is equivalent to $\Hasse^n_xg(p)\equiv 0$ for every $0\leq n<m$, the desired statement follows.
\end{proof}

We also use the following simple proposition that describes the composition rule of Hasse derivatives.
\begin{proposition}\label{prop:Hasse-composition}
    Let $g\in K[x,y]$ be any polynomial and let $m,n$ be two nonnegative integers.
    Then
    \[\Hasse_x^m\Hasse_x^ng = \binom{m+n}{m}\Hasse^{m+n}_xg.\]
\end{proposition}
\begin{proof}
    This follows immediately from the binomial theorem.
    To see this, note that
    \begin{align*}
        \Hasse^m_x\Hasse^n_xg =& [t_1^m](\Hasse^n_xg)(x+t_1,y) = [t_1^mt_2^n]g(x+t_1+t_2,y)\\ =& \binom{m+n}{n}[t^{m+n}]g(x+t,y) = \binom{m+n}{n}\Hasse^{m+n}_xg,
    \end{align*}
    where the second-to-last equality holds by the binomial theorem.
\end{proof}

Combining the two propositions gives the following proposition.
\begin{proposition}\label{prop:Hasse-decrease-mult}
    Let $g\in K[x,y]$ be any polynomial and $p\in K^2$ be any point.
    For any $n\in \ZZ_{\geq 0}$, we have
    \[\mu_{\Hasse^n_xg}^{\textup{hor}}(p)\geq \mu_g^{\textup{hor}}(p)-n.\]
\end{proposition}
\begin{proof}
    By \cref{prop:Hasse-is-multi}, it suffices to show that 
    \[\Hasse^i_x\Hasse^n_xg(p)=0\]
    for every $0\leq i < \mu_g^{\textup{hor}}(p)-n$.
    By \cref{prop:Hasse-composition} and \cref{prop:Hasse-is-multi}, we have that
    \[\Hasse^i_x\Hasse^n_xg(p)= \binom{i+n}{i}\Hasse^{n+i}_xg(p) = 0,\]
    as desired.
\end{proof}

We would also like to understand how Hasse derivatives behave when we factor a polynomial into factors.
This is answered by the following property.
\begin{proposition}\label{prop:Hasse-product-rule}
    Let $f_1,f_2$ be two polynomials in $K[x_1,\ldots, x_d]$ and let $n\in\NN$.
    Then
    \[\Hasse^n_{x_1}(f_1f_2) = \sum_{i=0}^{n}\left(\Hasse^i_{x_1}f_1\right)\left(\Hasse^{n-i}_{x_1}f_2\right).\]
    In general, if $f_1,\ldots, f_k$ are polynomials in $K[x_1,\ldots, x_d]$, then
    \[\Hasse^n_{x_1}(f_1\cdots f_k) = \sum_{n_1+\cdots+n_k=n}\prod_{i=1}^{k}\Hasse^{n_i}_{x_1}f_i.\]
\end{proposition}
\begin{proof}
    By definition
    \begin{align*}
        \Hasse^n_{x_1}(f_1f_2) =& [t^n]f_1(x_1+t,x_2,\ldots, x_d)f_2(x_1+t,x_2,\ldots,x_d)\\ =& \sum_{i=0}^{n}\left([t^i]f_1(x_1+t,x_2,\ldots, x_d)\right)\left([t^{n-i}]f_2(x_1+t,x_2,\ldots,x_d)\right)\\
        =& \sum_{i=0}^{n}\left(\Hasse^i_{x_1}f_1\right)\left(\Hasse^{n-i}_{x_1}f_2\right).
    \end{align*}
    The statement with multiple polynomials follows by induction.
    It also follows from the same argument.
\end{proof}

Using \cref{prop:Hasse-is-multi}, we also know how multiplicities behave when the polynomial is factored.

\begin{proposition}\label{prop:multi-product}
    Let $g_0,\ldots, g_{k-1}$ be polynomials in $K[x,y]$.
    Set $g = g_0\cdots g_{k-1}$.
    For any $p\in K^2$, we have
    \[\mu^{\textup{hor}}_g(p) = \sum_{i=0}^{k-1}\mu^{\textup{hor}}_{g_i}(p).\]
\end{proposition}
\begin{proof}
    Let $n\leq \sum_{i=0}^{k-1}\mu^{\textup{hor}}_{g_i}(p)$ be any nonnegative integer.
    Then
    \[\Hasse^n_{x}g(p) = \sum_{n_0+\cdots+n_{k-1}=n}\prod_{i=0}^{k-1}\Hasse^{n_i}_{x}g_i(p).\]
    Note that for each $(n_0,\ldots, n_{k-1})$ that sums to $n$, the summand is zero if there is some $i$ with $n_i<\mu^{\textup{hor}}_{g_i}(p)$ by \cref{prop:Hasse-is-multi}.
    Therefore we just need to keep the summands with $n_i\geq \mu^{\textup{hor}}_{g_i}(p)$ for all $i=0,\ldots,k-1$.
    This shows that $\Hasse^n_{x}g(p) = 0$ unless $n\geq\sum_{i=0}^{k-1}\mu^{\textup{hor}}_{g_i}(p)$.
    This already shows that $\mu^{\textup{hor}}_g(p) \geq \sum_{i=0}^{k-1}\mu^{\textup{hor}}_{g_i}(p)$.
    To show the reverse inequality, simply note that the only summand that survives for $n=\sum_{i=0}^{k-1}\mu^{\textup{hor}}_{g_i}(p)$, assuming it is finite, is when $n_i = \mu^{\textup{hor}}_{g_i}(p)<\infty$.
    Therefore
    \[\Hasse^n_{x}g(p) = \prod_{i=0}^{k-1}\Hasse^{n_i}_{x}g_i(p)\neq 0\]
    as each term in the product is nonzero by \cref{prop:Hasse-is-multi}.
\end{proof}

Finally, motivated by \cref{ex:why-it-breaks}, we would like to show that vanishing of certain Hasse derivatives implies that the polynomial does not depend on a variable.
This is done by the following proposition.

\begin{proposition}\label{prop:does-not-depend}
    Let $f\in K[x_1,\ldots, x_d]$ be a polynomial, and let $k$ be the smallest positive integer so that $\deg_{x_1} f<\fp^k$.
    If 
    \[\Hasse^{\fp^0}_{x_1}f=\cdots =\Hasse^{\fp^{k-1}}_{x_1}f\equiv 0,\]
    then $f$ does not depend on $x_1$.
\end{proposition}
\begin{proof}
    We will prove the contrapositive.
    Write $f = \sum_{n=0}^{\deg_{x_1}f}c_n(x_2,\ldots, x_d)x_1^n$.
    By the definition of Hasse derivative, we know that
    \[c_n(x_2,\ldots, x_d) = \Hasse^n_{x_1}f(0,x_2,\ldots, x_d)\]
    for any $n\in \ZZ_{\geq 0}$.
    
    Suppose that $f$ depends on $x_1$.
    Then there exists $n_*\in \NN$ upper bounded by $\deg_{x_1} f$ with $\Hasse^{n_*}_{x_1}f(0,x_2,\ldots, x_d)\not\equiv 0$.
    If $\fp=\infty$, set $v=0$.
    Otherwise, let $v$ be the largest nonnegative integer with $\fp^v\mid n_*$.
    It is clear that $v<k$.
    We may compute, using \cref{prop:Hasse-composition}, that
    \[\left(\Hasse^{n_*-\fp^v}_{x_1}\Hasse^{\fp^v}_{x_1}f\right)(0,x_2,\ldots,x_d) = \binom{n_*}{\fp^v}\Hasse_{x_1}^{n_*}f(0,x_2,\ldots, x_d)\neq 0\]
    as $\binom{n_*}{\fp^v}$ is nonzero: this is trivial if $\fp=\infty$, and this is a consequence of Lucas's theorem if $\fp<\infty$.
    As a consequence, we must have $\Hasse^{\fp^v}_{x_1}f\not\equiv 0$, as desired.
\end{proof}

\subsection{Proof of the improved bound}
With Hasse derivatives, we can now prove \cref{thm:main} using Lewko's framework.
The two lemmas that require bounds on the degrees of polynomials are \cref{lemma:bound-sum-mult} and \cref{lemma:good-slice-square-free}.
Unfortunately, \cref{lemma:good-slice-square-free} simply does not hold without the bound: when $\fp<\infty$, the polynomial $x^{\fp}-s$ is square-free, whereas for any fixed $a\in \FF$ we may write $x^{\fp}-a = (x-a^{1/\fp})^{\fp}$, which is not square-free in $\FF[x]$.
To fix this, we need to treat $s$ as an indeterminate and work over the field $\FF(s)$ instead.

We remark that working over the field $\FF(s)$ is morally the same as fixing $s$ to be a generic value in $\FF$.
They are most of the time equivalent, but unfortunately \cref{lemma:good-slice-square-free} witnesses a rare scenario where this equivalence fails.
Despite this, the readers may still find the geometric intuition in \cref{sec:lewko} helpful.

Let $K=\FF(s)$.
We can now generalize \cref{lemma:bound-sum-mult} to polynomials with higher degree while working over $K$.

\begin{lemma}\label{lemma:modified-bound-sum-mult}
    Let $g\in K[x,y]$ be a square-free polynomial with degree $d$.
    Let $k$ be the minimum positive integer so that $d<\fp^k$.
    Suppose that there are no factors of $g$ that only depend on $y$.
    Then
    \[\sum_{p\in K^2}\max\left\{\mu^{\textup{hor}}_g(p)-\left(1+\cdots+\fp^{k-1}\right),0\right\}\leq d(d-1).\]
\end{lemma}
\begin{proof}
    In this proof, all Hasse derivatives are with respect to $x$.
    For any irreducible factor $h$ of $g$, by \cref{prop:does-not-depend}, there exists some $i=0,\ldots,k-1$ so that $\Hasse^{\fp^i}h\not\equiv 0$.
    Let $i_h$ be the smallest such $i$.
    For any $i=0,\ldots,k-1$, let $g_i$ be the product of all irreducible factors $h$ of $g$ with $i_h=i$.
    It is clear that $g=g_0\cdots g_{k-1}$.

    We now fix some $i=0,\ldots, k-1$.
    The key claim is that $g_i$ and $\Hasse^{\fp^i}g_i$ are coprime.
    To see that, suppose for the sake of contradiction that $h$ is an irreducible common factor of $g_i$ and $\Hasse^{\fp^i}g_i$.
    Write $g_i=hq$.
    Then by \cref{prop:Hasse-product-rule},
    \[\Hasse^{\fp^i}g_i = \sum_{n=0}^{\fp^i}\left(\Hasse^nh\right)\left(\Hasse^{\fp^i-n}q\right).\]
    
    To compute this sum, we claim the summand is zero whenever $1\leq n < \fp^i$.
    This statement is vacuously true when $\fp=\infty$ as $i=0$, so we assume $\fp<\infty$ in this paragraph.
    Set $v$ to be the largest nonnegative integer with $\fp^v\mid n$.
    Then $v<i$.
    By the fact that $i_h=i$, we then have
    \[0\equiv \Hasse^{n-\fp^v}\Hasse^{\fp^v} h = \binom{n}{\fp^v}\Hasse^nh,\]
    showing that $\Hasse^nh\equiv 0$ as $\binom{n}{\fp^v}\neq 0$ by Lucas's theorem.
    
    Using the claim, we have
    \[\Hasse^{\fp^i}g_i = h\Hasse^{\fp^i}q+q\Hasse^{\fp^i}h.\]
    This shows that $h\mid q\Hasse^{\fp^i}h$ and so $h\mid \Hasse^{\fp^i}h$ as $g$ is square-free.
    By the definition of Hasse derivative, it is easy to see that $\deg \Hasse^{\fp^i}h\leq \deg h-\fp^i<\deg h$.
    This forces $\Hasse^{\fp^i}h\equiv 0$, which is a contradiction to the definition of $g_i$.
    Therefore $g_i$ and $\Hasse^{\fp^i}g_i$ are indeed coprime.

    We may now apply B\'ezout's theorem with multiplicities to $g_i$ and $\Hasse^{\fp^i}g_i$.
    For any $p\in K^2$, by \cref{prop:Hasse-decrease-mult}, we know that $g_i$ and $\Hasse^{\fp^i}g_i$ both have multiplicities at $p$ along the horizontal direction lower bounded by $\max\{\mu^{\textup{hor}}_{g_i}(p)-\fp^i,0\}$.
    This shows that the intersection multiplicity of $g_i$ and $\Hasse^{\fp^i}g_i$ is lower bounded by the same quantity, and so B\'ezout's theorem implies
    \[\sum_{p\in K^2}\max\{\mu^{\textup{hor}}_{g_i}(p)-\fp^i,0\}\leq \deg g_i\deg \Hasse^{\fp^i}g_i \leq \deg g_i\left(\deg g_i-\fp^i\right).\]

    By \cref{prop:multi-product}, we deduce
    \begin{align*}
        &\sum_{p\in K^2}\max\left\{\mu^{\textup{hor}}_g(p)-\left(1+\cdots+\fp^{k-1}\right),0\right\}\\ \leq& \sum_{i=0}^{k-1}\sum_{p\in K^2}\max\{\mu^{\textup{hor}}_{g_i}(p)-\fp^i,0\} \\\leq& \sum_{i=0}^{k-1}\deg g_i\left(\deg g_i-\fp^i\right)\leq d(d-1),
    \end{align*}
    as desired.
\end{proof}

We can now prove \cref{thm:main} using \cref{lemma:modified-bound-sum-mult} and the ingredients in \cref{sec:lewko}.
As we are working over $K$ instead of fixing $s$, we will often treat polynomials in $\FF[x,y,s]$ as polynomials in $K[x,y]$.
Using the lemmas in \cref{sec:lewko} thus sometimes requires reformulations, and we will mention those when they show up.
\begin{proof}[Proof of \cref{thm:main}]
    Recall that we have a set of points $\cP$ and a set of lines $\cL$ in $\FF^2$, and we assume that $k$ is a positive integer such that $\fp^{2k-2}\leq \abs{\cL}<\fp^{2k}$.
    Let $c_k = 1+\cdots+\fp^{k-1}=\Theta(\fp^{k-1})$.
    We will show that for any real $1\leq d<\fp^k$, we have
    \begin{equation}\label{eq:modified-lewko-st-with-param}
        I(\cP,\cL)\leq c_k\abs{\cP}+\abs{\cL}+d(d-1)+\frac{2\abs{\cP}\abs{\cL}}{d},
    \end{equation}
    which immediately implies \cref{thm:main} with $d = \abs{\cP}^{1/3}\abs{\cL}^{1/3}$ if $\abs{\cP}^{1/3}\abs{\cL}^{1/3}<\fp^k$, or with $d$ arbitrarily close to $\fp^k$ otherwise.

    Take $e = \lfloor 2\abs{\cL}/d\rfloor$, and lift the lines in $\cL$ to $\Tilde{\cL}$ in $\FF^3$.
    As in the proof of \cref{thm:lewko-st-with-param}, we have $(\lfloor d\rfloor+2)(e+1)>2\abs{\cL}$, and so there is a polynomial $f$ given by \cref{lemma:param-counting} with parameters $(\lfloor d\rfloor,e)$ that vanishes identically on $\Tilde{\cL}$.
    By removing repeated factors, we may assume that $f$ is square-free.
    We still let $\cP_{\textup{ex}}$ be the set of points in $\cP$ for which $f$ vanishes identically on $R_p$.

    For any $p\in \cP$, let $r(p)$ be the number of lines in $\cL$ through $p$.
    Consider the embedding $\phi:\FF^2\to K^2$ sending $p=(p_x,p_y)$ to $(p_x,p_y-sp_x)$.
    Also set $g\in K[x,y]$ where $g(x,y) = f(x,y+sx,s)$.
    We claim that
    \[\mu^{\textup{hor}}_g(\phi(p))\geq r(p)-e\]
    for every $p\in \cP$.
    To see this, note that \cref{lemma:incidence-to-mult} shows that $f(p_x+t,p_y+st,s)$, when expressed as a polynomial in $K[t]$, satisfies that
    \[[t^i]f(p_x+t,p_y+st,s)=0\quad\forall 0\leq i < r(p)-e.\]
    Note that $f(p_x+t,p_y+st,s) = g(p_x+t,p_y-sp_x)$.
    Therefore the left hand side can be simply written as $\Hasse^i_xg(\phi(p))$, showing that $\mu^{\textup{hor}}_g(\phi(p))\geq r(p)-e$ by \cref{prop:Hasse-is-multi}.

    Now let $h$ be the product of factors of $g$ that only depend on $y$, and let $g' = g/h$.
    We claim that if $h(\phi(p))=0$, then $p\in \cP_{\textup{ex}}$.
    Indeed, if $h(\phi(p))=0$, then we know that $g(p_x+t,p_y-sp_x)=0$ holds for all $t\in \FF$.
    Therefore $f(p_x+t,p_y+st,s)=0$ holds for all $s,t\in \FF$, and so $p\in \cP_{\textup{ex}}$ by definition.
    As a consequence, by \cref{prop:multi-product}, we have $\mu^{\textup{hor}}_{g'}(\phi(p)) = \mu^{\textup{hor}}_g(\phi(p))\geq r(p)-e$ for every $p\in \cP\backslash \cP_{\textup{ex}}$.

    We are ready to bound $I(\cP\backslash \cP_{\textup{ex}},\cL)$.
    By the bound on $\mu^{\textup{hor}}_{g'}(\phi(p))$, we have
    \[I(\cP\backslash \cP_{\textup{ex}},\cL)=\sum_{p\in \cP\backslash \cP_{\textup{ex}}}r(p)\leq (e+c_k)\left(\abs{\cP}-\abs{\cP_{\textup{ex}}}\right)+\sum_{p\in \cP\backslash \cP_{\textup{ex}}}\max\left\{\mu^{\textup{hor}}_{g'}(\phi(p))-c_k,0\right\}.\]
    Note that $g'$ has no factors that depend only on $y$ and $\deg g'\leq \deg g\leq \deg_{x,y}f\leq d$.
    It is also square-free in $K[x,y]$ as $f$ is square-free in $\FF[x,y,s]$ and thus in $K[x,y]$.
    By \cref{lemma:modified-bound-sum-mult}, the second sum can be bounded by $d(d-1)$.
    Therefore
    \[I(\cP\backslash \cP_{\textup{ex}},\cL)=\sum_{p\in \cP\backslash \cP_{\textup{ex}}}r(p)\leq (e+c_k)\left(\abs{\cP}-\abs{\cP_{\textup{ex}}}\right)+d(d-1).\]
    Finally, we still have $I(\cP_{\textup{ex}},\cL)\leq (e+1)\abs{\cP_{\textup{ex}}}+\abs{\cL}$ from the proof of \cref{thm:lewko-st-with-param}.
    Summing up, we get
    \[I(\cP,\cL)\leq (e+c_k)\abs{\cP}+\abs{\cL}+d(d-1).\]
    Since $e\leq 2\abs{\cL}/d$, \cref{eq:modified-lewko-st-with-param} immediately follows, as desired.
\end{proof}

We now briefly discuss how Miao and Xie's proof can be modified.
Recall that in \cref{lemma:mx-descend-poly}, we have $h\in \FF[x,y,s]$ with $(\partial_x+s\partial_y)h\equiv 0$.
After setting $g(t,b,a) = h(t,at+b,a)$, we would like to say that $\partial_tg(t,b,a) = ((\partial_x+s\partial_y)h)(t,at+b,a)\equiv 0$ and so $g$ does not depend on $t$.
As discussed earlier, when $\deg_t g\geq \fp$ we need to replace this with \cref{prop:does-not-depend}.
Therefore we now need $\Hasse^{\fp^0}_tg=\cdots= \Hasse^{\fp^{k-1}}_tg\equiv 0$ as input, where $k$ is the smallest positive integer with $\fp^k>\deg_t g$.
To this end, let $D_ih(t,at+b,a) = \Hasse^{\fp^i}_tg(t,b,a)$; equivalently,
\[D_ih(x,y,s) = [t^{\fp^i}] h(x+t,y+st,s).\]
Now \cref{lemma:mx-descend-poly} can be generalized to any $h$ with $\deg_{x,y}h<\fp^k$ as long as $D_0h=\cdots= D_{k-1}h\equiv 0$.
\begin{lemma}\label{lemma:modified-mx-descend-poly}
    Let $h\in\FF[x,y,s]$ be a polynomial, and let $k$ be the smallest positive integer with $\deg_{x,y}h<\fp^k$.
    If $D_0h=\cdots=D_{k-1}h\equiv 0$, then $\cL(h)$ is precisely the zero locus of the polynomial $\underline{h}(a,b) \coloneqq h(0,b,a)$.
    Moreover, $\deg \underline{h}= \deg_s h$.
\end{lemma}

The main proof has to be modified accordingly.
Starting with a nonzero square-free polynomial $f$ that vanishes on all lines in $\Tilde{\cL_r}$, the original strategy is to use the fact that $Df$ also vanishes identically on the lines in $\Tilde{\cL_r}$ and to apply \cref{lemma:mx-improved-bezout} to an appropriate factor of $f$ and its derivative.
However, for some irreducible factor $h$ of $f$, it is now possible that $D_0h\equiv 0$ but $D_ih\not\equiv 0$ for some larger $i$.
To handle all possible situations, let $f_{\textup{ex}}$ be the product of irreducible factors $h$ of $f$ with $D_0h=\cdots=D_{k-1}h\equiv 0$; also for $i=0,\ldots, k-1$, let $f_i$ be the product of irreducible factors $h$ of $f$ with $D_0h=\cdots=D_{i-1}h\equiv 0$ but $D_ih\not\equiv 0$.
The lifted lines contained in $f_{\textup{ex}}$ can be dealt with using the exact same argument once \cref{lemma:mx-descend-poly} is replaced with \cref{lemma:modified-mx-descend-poly}.
On the other hand, using \cref{prop:Hasse-product-rule}, one can check that $f_i$ and $D_if_i$ are coprime.
Moreover, one can check that $\deg_{x,y}D_if_i\leq \deg_{x,y}f_i-\fp^i$ and $\deg_sD_if_i\leq \deg_sf_i+\fp^i$.
By applying \cref{lemma:mx-improved-bezout}, we see that the number of lifted lines contained in $f_i$ is at most
\[\deg_sf_i\left(\deg_{x,y}f_i-\fp^i\right)+\deg_{x,y}f_i\left(\deg_sf_i+\fp^i\right)\leq 2\deg_{x,y}f_i\deg_sf_i+\fp^i\deg_{x,y}f_i.\]
With the modification, one can prove the following variant of \cref{thm:mx-bound-rich-line}.
\begin{theorem}\label{thm:modified-mx-bound-rich-line}
    Let $\cP$ be a set of points in $\FF^2$, and let $k$ be a positive integer.
    Assume that $r$ is a positive integer satisfying $r\geq \max\left\{2,\frac{12\abs{\cP}}{\fp^k}\right\}$.
    Let $\cL_r$ be the set of $r$-rich lines for $\cP$.
    Then
    \[\abs{\cL_r}\lesssim \frac{\abs{\cP}^2}{r^3}+\frac{\fp^{k-1}\abs{\cP}}{r}.\]
\end{theorem}
We remark that this bound on the number of rich lines matches what \cref{thm:main} (or more precisely \cref{eq:modified-lewko-st-with-param}) gives: indeed, we have
\[r\abs{\cL_r}\leq I(\cP,\cL_r)\leq \left(1+\cdots+\fp^{k-1}\right)\abs{\cP}+\abs{\cL_r}+3\abs{\cP}^{2/3}\abs{\cL_r}^{2/3}+\frac{2\abs{\cP}\abs{\cL_r}}{\fp^k},\]
showing that $\abs{\cL_r}\lesssim \frac{\abs{\cP}^2}{r^3}+\frac{\fp^{k-1}\abs{\cP}}{r}$ for every $r\geq \max\left\{2,\frac{4\abs{\cP}}{\fp^k}\right\}$.
Unfortunately, when we sum \cref{thm:modified-mx-bound-rich-line} up to get an incidence bound, the $\fp^{k-1}\abs{\cP}/r$ term may dominate in the entire range of $\frac{12{\abs{\cP}}}{\fp^k}\leq r< \frac{12\abs{\cP}}{\fp^{k-1}}$.
This term alone already sums up to $\left(\fp^{k-1}\log\fp \right)\abs{\cP}$ within that range.
Therefore summing \cref{thm:modified-mx-bound-rich-line} only gives a slightly weaker bound
\[I(\cP,\cL)\lesssim \left(\fp^{k-1}\log\fp \right)\abs{\cP}+\abs{\cL}+\abs{\cP}^{2/3}\abs{\cL}^{2/3}+\frac{\abs{\cP}\abs{\cL}}{\fp^k}.\]

\subsection{Tightness of the improved bound}\label{subsec:tightness}
In this subsection, we show that for any given $m,n\in \NN$ with $m\geq n$, if $\fp^{2(k-1)}\leq m<\fp^{2k}$, then we may find a set of at most $n$ points $\cP$ and at most $m$ lines $\cL$ in $\FF^2$ with
\[I(\cP,\cL)\gtrsim \fp^{k-1}n+m+m^{2/3}n^{2/3}+\frac{mn}{\fp^k}.\]

If $\fp^{k-1}n$ is the dominating term, then we know that $\fp<\infty$ and $k\geq 2$.
Note that $\FF_{\fp^{k-1}}^2$ has precisely $\fp^{2k-2}$ non-vertical lines.
Let $\cL$ be the union of $\lfloor m/\fp^{2k-2}\rfloor$ generic translations of the lines in $\FF_{\fp^{k-1}}^2$.
Let $\Tilde{\cP}$ be the union of the same $\lfloor m/\fp^{2k-2}\rfloor$ generic translations of the points in $\FF_{\fp^{k-1}}^2$.
Then $\abs{\Tilde{\cP}}=\abs{\cL}\leq m$ and each point in $\Tilde{\cP}$ is on precisely $\fp^{k-1}$ lines in $\cL$.
Let $\cP=\Tilde{\cP}$ if $n\geq \abs{\Tilde{\cP}}$, and otherwise take $\cP$ to be any $n$-subset of $\Tilde{\cP}$.
Since $\abs{\cL}\gtrsim m$, we know that $\abs{\cP}\gtrsim n$ and so $I(\cP,\cL)\gtrsim \fp^{k-1}n$, as desired.

If $m$ is the dominating term, we can simply take a point and $m$ lines through the point, showing that $I(\cP,\cL)=m$, as desired.

Now suppose that $m^{2/3}n^{2/3}$ is the dominating term.
More precisely, we assume $m^{2/3}n^{2/3}\geq 2m$, $m^{2/3}n^{2/3}\geq \fp^{k-1}n$ and $m^{2/3}n^{2/3}\geq \frac{2mn}{\fp^{k}}$.
We will follow the standard construction by Elekes \cite{Elekes02} (see \cite{SdZ17} for example), although we need to make a small modification to extend it to all regimes.
This modification was suggested by Dmitrii Zakharov and independently ChatGPT-6 Astra.

Pick any $\theta\in \FF\backslash \FF_{\fp^{k-1}}$.
For any positive integers $a,b$ with $ab< \fp/2$, we set
\[\cP = \left\{(x,y_1+y_2\theta): x\in [a], y_1\in \FF_{\fp^{k-1}},y_2\in[2ab]\right\}\]
and
\[\cL = \left\{\left\{(x,y)\in \FF^2:y=(c_1+c_2\theta)x+(d_1+d_2\theta)\right\}:c_1,d_1\in \FF_{\fp^{k-1}}, c_2\in [b], d_2\in [ab]\right\}.\]
We remark that if $k=1$, by $\FF_{\fp^0}$ we mean the set $\{0\}$.
Then we have $\abs{\cP} = 2\fp^{k-1}a^2b$ and $\abs{\cL} = \fp^{2k-2}ab^2$.
We will choose $a,b$ later so that $\abs{\cP},\abs{\cL}$ are bounded by $n,m$, respectively.
It is easy to see that each line in $\cL$ contains precisely $a$ points in $\cP$ as $x$ ranges through $[a]$.
Therefore
\[I(\cP,\cL)\geq \fp^{2(k-1)}a^2b^2\gtrsim \abs{\cP}^{2/3}\abs{\cL}^{2/3},\]
and it suffices to verify that $\abs{\cP}\gtrsim n$ and $\abs{\cL}\gtrsim m$.

Take $a=\left\lfloor\sqrt[3]{\frac{n^2}{4m}}\right\rfloor$ and $b=\left\lfloor\sqrt[3]{\frac{2m^2}{\fp^{3(k-1)}n}}\right\rfloor$.
First, we would like to verify that $a,b\geq 1$ and $ab<\fp/2$.
We have that $\frac{n^2}{m} = \frac{(m^{2/3}n^{2/3})^{3}}{m^{3}}\geq 8$, showing that $a\geq 1$.
We also have $\frac{m^2}{\fp^{3(k-1)}n} = \frac{(m^{2/3}n^{2/3})^3}{(\fp^{k-1}n)^3}\geq 1$, showing that $b\geq 1$.
Finally, if $\fp=\infty$ then we clearly have $ab<\fp/2$.
If $\fp<\infty$, we can compute $\frac{mn}{\fp^{3(k-1)}} = \frac{\fp^3\cdot (\fm n/\fp^k)^3}{(m^{2/3}n^{2/3})^{3}}\leq (\fp/2)^3$, showing that $ab<\fp/2$.
Therefore this choice of $a,b$ is valid.
Moreover, we have that 
\[\sqrt[3]{\frac{n^2}{4m}}\geq a\gtrsim \sqrt[3]{\frac{n^2}{m}}\textup{ and }\sqrt[3]{\frac{2m^2}{\fp^{3(k-1)}n}}\geq b\gtrsim \sqrt[3]{\frac{m^2}{\fp^{3(k-1)}n}}.\]
Therefore we can compute
\[n\geq 2\fp^{k-1}a^2b= \abs{\cP}\gtrsim n\textup{ and }m\geq \fp^{2k-2}ab^2=\abs{\cL}\gtrsim m,\]
as desired.

Finally, suppose that $\frac{mn}{\fp^k}$ is the dominating term.
In particular, we suppose that $\frac{mn}{\fp^k}\geq m$.
Then $\fp^k\leq n\leq m$ and so $m/\fp^k\geq 1$.
Let $a = \lfloor m/\fp^k\rfloor< \fp^k$ and let $A$ be any $a$-subset of $\FF_{\fp^k}$.
We can take 
\[\cL = \{\{(x,y)\in \FF^2:y=cx+d\}:c\in A, d\in \FF_{\fp^k}\},\]
which clearly satisfies $\abs{\cL}\leq m$.
Now take $\cP$ to be any $n$-subset of $\FF_{\fp^k}^2$, which is valid as $n\leq m\leq \fp^{2k}$.
Since any point in $\FF_{\fp^k}^2$ is contained in precisely $a$ lines in $\cL$, we have
\[I(\cP,\cL)\geq an \gtrsim \frac{mn}{\fp^k},\]
as desired.

\section{Concluding remarks}\label{sec:conclude}
\subsection{Point-line incidences over finite fields}
As shown in \cref{subsec:tightness}, \cref{thm:main} gives a tight upper bound on the number of incidences between points and lines in $\FF^2$ when $\FF$ contains the appropriate subfields.
This shows that when we fix $\FF = \FF_q$ for some $q=\fp^{\alpha}$ with $\alpha\geq 3$, the incidence bound might not be tight in all regimes.
Indeed, Vinh's bound \cite{Vinh11} shows that \cref{thm:main} is not always tight: it states that
\[\abs{I(\cP,\cL)-\frac{\abs{\cP}\abs{\cL}}{q}}\lesssim q^{1/2}\abs{\cP}^{1/2}\abs{\cL}^{1/2},\]
and so
\[I(\cP,\cL)\lesssim \frac{\abs{\cP}\abs{\cL}}{q}+q^{1/2}\abs{\cP}^{1/2}\abs{\cL}^{1/2}.\]
Now when $q=\fp^{\alpha}$ with $\alpha \geq 3$, let $k$ be a positive integer with $\alpha/2<k<\alpha$ and set $\abs{\cP}=\abs{\cL} = \fp^{2k}$.
Then Vinh's bound gives $I(\cP,\cL)\lesssim \fp^{4k-\alpha}+\fp^{2k+\frac{\alpha}{2}}$ while \cref{thm:main} gives $I(\cP,\cL)\lesssim \fp^{3k}$.
A simple computation shows that Vinh's bound is always better than \cref{thm:main} at least by a factor of $\fp^{1/2}$.
It would be interesting to see whether it is possible to get an incidence bound tight up to a constant even over $\FF_q$ with $q=\fp^{\alpha}$ and $\alpha \geq 3$.
\subsection{Replacing polynomial partitioning}
One of the proofs of the Szemer\'edi--Trotter theorem is to apply the Guth--Katz polynomial partitioning theorem (see \cite[Section 3]{KMS12}), and both proofs replace polynomial partitioning with the new argument that still uses polynomial methods but no longer uses topology of the underlying field.
It is natural to go through the list of problems where polynomial partitioning was used and ask if the new framework can be used to replace polynomial partitioning in their proofs.
The most natural candidate is the Guth--Katz point-line incidence theorem in $\RR^3$ \cite[Theorem 4.5]{GK15}, for which the polynomial partitioning technique was first devised and applied.

One can still lift the configuration from $\FF^3$ to $\FF^5$ by lifting any line of the form $(t,at+b,ct+d)$ to $(t,at+b,ct+d,a,c)$.
However, each point is now lifted to a $2$-plane instead of a line, and having many incidences is no longer sufficient to force vanishing on the entire $2$-plane.
This breaks \cref{lemma:incidence-to-mult}, which was crucial to Lewko's proof.
In general, it seems more difficult to generalize Lewko's proof to incidence problems where there is a $2$-parameter family of curves going through each point, whereas it is plausible that the method generalizes if there is only a $1$-parameter family of curves through each point.
On the other hand, it seems very plausible that Miao and Xie's proof could generalize to some extent.
\section*{Acknowledgment}
The author thanks Noah Kravitz, Olivine Silier, Jonathan Tidor, Dmitrii Zakharov and Daniel Zhu for conversations that eventually led to this note.
The author is particularly grateful to Jonathan Tidor for inspiring the diagrams in this note.

\bibliographystyle{plain}
\bibliography{bib.bib}
\end{document}